\documentclass[microtype]{gtpart}     
\DeclareMathOperator{\Rm}{\textup{Rm}}
\DeclareMathOperator{\R}{\textup{R}}
\DeclareMathOperator{\vol}{\textup{Vol}}
\DeclareMathOperator{\Ric}{\textup{Ric}}

\title{Noncollapsed degeneration of Einstein $4$-manifolds, I}

\author{Tristan Ozuch}
\givenname{Tristan}
\surname{Ozuch}
\address{MIT, Dept. of Math., 77 Massachusetts Avenue, Cambridge, MA 02139-4307.}
\email{ozuch@mit.edu}
\urladdr{https://tristanozuch.github.io/}


\arxivreference{}

\newtheorem{thm}{Theorem}[section]    
\newtheorem{lem}[thm]{Lemma}          
\theoremstyle{definition}
\newtheorem{defn}[thm]{Definition}    
\newtheorem{rem}{Remark}             
\newtheorem{quest}{Question}

\newtheorem{exmp}{Example}

\newtheorem{prop}{Proposition}

\newtheorem{cor}{Corollary}
\begin{document}

\begin{abstract}    
A theorem of Anderson and Bando-Kasue-Nakajima from 1989 states
that to compactify the set of normalized Einstein metrics with a lower bound on the volume and an upper bound on the diameter in the Gromov-Hausdorff sense, one has to add singular spaces called Einstein orbifolds, and the singularities form as blow-downs of Ricci-flat ALE spaces.

This raises some natural issues, in particular: can all Einstein orbifolds be Gromov-Hausdorff limits of smooth Einstein manifolds? Can we describe more precisely the smooth Einstein metrics close to a given singular one?

In this first paper, we prove that Einstein manifolds sufficiently close, in the Gromov-Hausdorff sense, to an orbifold are actually close to a gluing of model spaces in suitable weighted Hölder spaces. The proof consists in controlling the metric in the neck regions thanks to the construction of optimal coordinates.

This refined convergence is the cornerstone of our subsequent work on the degeneration of Einstein metrics or, equivalently, on the desingularization of Einstein orbifolds in which we show that all Einstein metrics Gromov-Hausdorff close to an Einstein orbifold are the result of a gluing-perturbation procedure. This procedure turns out to be generally obstructed, and this provides the first obstructions to a Gromov-Hausdorff desingularization of Einstein orbifolds. 
\end{abstract}

\maketitle

\tableofcontents

\section*{Introduction}

A fundamental question in geometry and topology is the following: given a \emph{topology} (a differentiable manifold, $M$), is there an optimal \emph{geometry} (a Riemannian metric, $g$) with this topology?

An Einstein metric, $g$ satisfies, for some real $\Lambda$, the equation
$$\Ric(g)=\Lambda g,$$
where $\Ric$ is the Ricci curvature. These metrics are considered optimal for the homogeneity of their Ricci curvature and as critical points of the Einstein-Hilbert functional with fixed volume: $g\mapsto\int_M \R_g dvol_g$, where $\R$ is the scalar curvature, which is the spatial and directional mean value of the sectional curvatures.

In dimension $2$ and $3$ these metrics have constant sectional curvatures and are therefore well understood as they only have $3$ different local behavior depending on the sign of the curvature (spherical, flat and hyperbolic). Their understanding was crucial for $2$-dimensional geometry and topology thanks to the uniformization Theorem and in $3$-dimensional geometry topology thanks to Thurston's geometrization.

In dimension $4$, Einstein metrics are moreover optimal as minimizers of the $L^2$-norm of Riemann curvature tensor: $g\mapsto \int_M |\Rm_g|^2dvol_g$. Indeed, By Chern-Gauss-Bonnet formula, we have
$$ \int_M |\Rm_g|^2dv_g = 8\pi^2 \chi(M) +\int_M |\Ric^0_g|^2dv_g,$$
where $\chi$ is the Euler characteristic and where $\int_M |\Ric^0_g|^2dv_g\geqslant 0$ vanishes if and only if $g$ is Einstein. Notice that in this case, the quantity $\int_M |\Rm_g|^2dv_g = 8\pi^2 \chi(M)$ is purely topological.

But from dimension $4$, the Einstein condition also becomes flexible as it does not imply that the sectional curvatures are constant anymore. It is actually so flexible that Einstein metrics can develop singularities. One major goal for $4$-dimensional geometry is therefore to understand the set of Einstein metrics and how they can degenerate, that is to compactify the set of Einstein metrics. In this paper, we will be interested in Einstein $4$-manifolds which are noncollapsed, which means that they have their volume bounded from below. The basic tool for this kind of compactification question is Gromov's compactness theorem \cite{gro} from which the usual goals are to obtain informations about the possible limit spaces, and to understand if the convergence happens in a stronger sense than for the Gromov-Hausdorff distance. 

The description of the limit spaces was given by Anderson and Bando-Kasue-Nakajima. 
\begin{thm}[\cite{and,bkn}]\label{dégénérescence dim 4}
 Let $(M^4_i,g_i)_{i\in\mathbb{N}}$ be a sequence of Einstein $4$-manifolds satisfying,
\begin{enumerate}
\item the diameters of the $(M^4_i,g_i)$ are uniformly bounded,
\item the volumes of the $(M^4_i,g_i)$ are uniformly bounded from below,
\item the integrals $\mathcal{E}_i:= \int_{M^4_i}|\textup{Rm}(g_i)|_{g_i}^2 dv_{g_i}$ are uniformly bounded.
\end{enumerate}
Then, there exists a subsequence $(M^4_{i},g_{i})$ converging to an Einstein orbifold with isolated singularities (see Definition \ref{orb Ein}), $(M^4_\infty,g_\infty)$.
 
 Moreover, for any singular point $p_\infty\in M_\infty$, there exists a sequence $(p_{i})_{i\in \mathbb{N}}$ of points of $ M_i^4$ with $\frac{1}{t_{i}}:= |\textup{Rm}(g_i)|(p_{i})\to\infty$ such that $\left(M^4_{i},\frac{g_{i}}{t_{i}},p_{i}\right)$ converges in the pointed Gromov-Hausdorff sense to a Ricci-flat asymptotically locally Euclidean (ALE) manifold, see Definition \ref{def orb ale}.
\end{thm}

\begin{rem}
The first hypothesis prevents the formation of cusps, and the second prevents our space to collapse to a lower dimensional space. The third hypothesis, which is actually a consequence of the two others by \cite[Theorem 1.13]{cn}, is a topological condition by Chern-Gauss-Bonnet formula.
\end{rem}

This theorem leaves several questions open, in particular the following one. 
\begin{quest}[{\cite[7.I.]{andsurv}}]\label{quest which orbifolds}
    "It has long been an open question whether Einstein orbifold metrics can be resolved to smooth Einstein metrics in the Gromov-Hausdorff topology. The idea here would be to reverse the process of formation of orbifold singularities described in [Theorem \ref{dégénérescence dim 4}]"
\end{quest}
For example, a natural question is whether or not \emph{any} Einstein orbifold can be desingularized by Einstein metrics, say with an expected topology. 

\subsection*{Desingularization of Einstein $4$-orbifolds and obstructions}

A natural technique to desingularize Einstein orbifolds is a gluing-perturbation procedure: we glue Ricci-flat ALE manifolds to the singularities of the orbifold to obtain an approximate Einstein metric, and then try to perturb it into an actual Einstein metric. The existence of such desingularizations in the \emph{real} Einstein context, and therefore a partial answer to Anderson's question was proven by Biquard for nondegenerate (i.e. which do not have any $L^2$-infinitesimal Einstein deformation) asymptotically hyperbolic Einstein manifolds with one singularity $\mathbb{R}^4\slash\mathbb{Z}_2$. This is realized by a particular gluing procedure of an Eguchi-Hanson metric to the singularity, see \cite[Theorem 0.1]{biq1}. 

Strikingly, this particular desingularization is only possible if an \emph{obstruction} is satisfied: denoting $\mathbf{R}$ the Riemannian curvature seen as an endomorphism on $2$-forms, the orbifold must satisfy $\det \mathbf{R} = 0$ at its singular point. This obstruction is restrictive, and for example not satisfied by hyperbolic orbifolds (which are rigid) which therefore cannot be desingularized by this particular procedure.

\subsection*{Description of the degeneration in suitable weighted Hölder spaces}

The main goal of this series of paper is to prove that the obstruction of \cite{biq1} holds under much less assumptions. The least satisfactory one is that the metrics of the sequence have to come from a particular gluing-perturbation procedure. It even turns out that a convergence in the function spaces of \cite{biq1} cannot be true in general if one drops any of the following assumption:
\begin{enumerate}
    \item the orbifold is rigid (it does not admit infinitesimal Einstein deformations),
    \item there is only one singular point,
    \item the singularity is $\mathbb{R}^4\slash\mathbb{Z}_2$.
\end{enumerate}
The reasons are the following:
\begin{enumerate}
    \item the convergence speed towards the orbifold could be arbitrarily slow compared to the singularity formation,
    \item different singularities may form at different speed,
    \item there might be trees of singularities forming.
\end{enumerate}
In this paper, we will prove the cornerstone result of our program which states that if an Einstein metric is sufficiently close to an Einstein orbifold in the Gromov-Hausdorff sense, then it is actually close to a glued metric in some weighted Hölder norms bounded on symmetric $2$-tensors decaying in the neck regions. These spaces are necessarily different from \cite{biq1} by the above reasons. The main result is the following theorem.

\begin{thm}\label{proximité désing naive}
	Let $D_0,v_0>0$, $l\in \mathbb{N}$, then, there exists $\delta = \delta(D_0,v_0,l) >0$ such that if $(M,g^\mathcal{E})$ is an Einstein manifold satisfying
	\begin{itemize}
		\item the volume is bounded below by $v_0>0$,
		\item the diameter is bounded by $D_0>0$,
		\item the Ricci curvature is bounded $|\Ric|\leq 3$.
	\end{itemize}
	and for which there exists an Einstein orbifold $(M_o,g_o)$ with
	$$d_{GH}\big((M,g^\mathcal{E}),(M_o,g_o)\big)\leqslant \delta,$$
	then, $(M,g^\mathcal{E})$ is the result of a gluing-perturbation procedure detailed in \cite{ozu2}.
\end{thm}
Notice that we are exactly in the context of Theorem \ref{dégénérescence dim 4}, and that no assumption is made about the possible singularity models, the number of singular points or the orbifold metric. The goal of the present paper is the construction of suitable coordinates in which we will compare our metrics and prove a decay in the neck regions. 

Theorem \ref{dégénérescence dim 4} actually implies that there is a satisfactory $C^\infty$-convergence on compact regions of the orbifold and of the Ricci-flat ALE spaces (without their singular points) appearing. Our main concern here will be to analyze the neck regions linking them, and to prove that there is convergence in a weighted $C^\infty$-sense. The main idea is to foliate these neck regions by constant mean curvature hypersurfaces thanks to which we will construct coordinates in which we will finally control our Einstein metrics. 

\subsection*{Obstructions to the Gromov-Hausdorff desingularization of Einstein orbifolds}

In the compact case, we do not expect to desingularize general Einstein orbifolds by smooth Einstein metrics as there are infinitely many global (expected) obstructions, but we can identify obstructions to such a desingularization. In particular, we will identify the only local obstruction, $\det \mathbf{R} = 0$ at the singular points, to desingularizing orbifolds for a large class of manifolds.

There is a well-known family of Ricci-flat ALE spaces called \emph{gravitational instantons} which have been classified by \cite{kro} and the Kähler Ricci-flat ALE spaces which are quotients of gravitational instantons have been classified in \cite{suv}. It is a famous conjecture that all Ricci-flat ALE spaces are Kähler. One of our main results is an obstruction to the desingularization by this conjecturally exhaustive list of candidates.

\begin{thm}[{\cite{ozu2}}]\label{thm res min}
	Let $(M_i,g_i)$ be a sequence of Einstein manifolds converging in the Gromov-Hausdorff sense to an Einstein orbifold $(M_o,g_o)$, and assume that every singularity blow-up is a Kähler Ricci-flat ALE orbifold in the positive orientation.
	
	Then, at any singular point $p$ of $(M_o,g_o)$, we have
	$$\det \mathbf{R}_{g_o}(p) =0.$$
\end{thm}

The condition on the singularity model can be reformulated as a topological assumption by \cite{nak} and we have the following example.
\begin{exmp}[{\cite{ozu2}}]
    Consider $\mathbb{S}^4\slash\mathbb{Z}_2$ the Einstein orbifold with two $\mathbb{R}^4\slash\mathbb{Z}_2$ singularities obtained as the quotient of $\mathbb{S}^4$ by $\{\pm\}$ in a global geodesic chart. Then, there exists a differentiable manifold $M^4$, such that for any $1\leqslant p<\infty$, we can construct a sequence of metrics $(M,g_i)$ with both $$\|\Ric(g_i)-3g_i\|_{L^p(g_i)}\to 0, \text{ and }\Ric(g_i)\geqslant 3 g_i, $$
    while $$(M,g_i)\xrightarrow[]{GH} (\mathbb{S}^4\slash\mathbb{Z}_2,g_{\mathbb{S}^4\slash\mathbb{Z}_2}),$$
    but there \emph{does not} exist any sequence of Einstein metrics with
    $$\Ric(g_i)=3g_i,$$ and $$(M,g_i)\xrightarrow[]{GH} (\mathbb{S}^4\slash\mathbb{Z}_2,g_{\mathbb{S}^4\slash\mathbb{Z}_2}).$$
\end{exmp}

Let us finally mention one more application in the context of spin manifolds.

\begin{thm}[{\cite{ozu2}}]
    Let $(M_i,g_i)$ be a sequence of \emph{spin} Einstein $4$-manifolds converging in the Gromov-Hausdorff sense to an Einstein orbifold $(M_o,g_o)$. Then, $(M_o,g_o)$ is spin, and at each singular point with group in $SU(2)$, we have $$\det \mathbf{R}_{g_o} = 0.$$
\end{thm}

\subsection*{Outline of the paper}

In Section 1, we start by giving the principal needed definitions. We then precise the convergence in Theorem \ref{dégénérescence dim 4} thanks to the introduction of glued metrics which we call \emph{naïve desingularizations}. The goal of this article is to prove that they approximate Einstein metrics which are Gromov-Hausdorff close to an orbifold in the sense of appropriate weighted Hölder norms. We finally identify the neck regions in Einstein manifolds which are Gromov-Hausdorff close to an orbifold. The rest of the paper will be focused on these regions, and in Section 2, we construct some first coordinates in these neck regions by $\epsilon$-regularity.

In Section 3, we motivate the use of constant mean curvature hypersurfaces and prove that they are well controlled by the ambient geometry. 

We develop a perturbation technique in Section 4 which we use in Section 5 to foliate the neck regions by constant mean curvature hypersurfaces by perturbation of the spheres of the first coordinates of Section 2. 

In Section 6, we finally construct coordinates based on the foliation of Section 5 and use them to control the metric at the scale of the curvature which decays by the results of \cite{ban}.

\subsection*{Acknowledgements}

I would like to thank my PhD advisor, Olivier Biquard, for introducing me to the questions motivating this series of articles and for his constant support and encouragement.

\section{Trees of singularities and naïve desingularizations}

Let us start by giving some definitions and explaining why we can reduce our situation to the study of neck regions.

\subsection{Einstein orbifolds and Ricci-flat ALE orbifolds}

The singular limit spaces in Theorem \ref{dégénérescence dim 4} are Einstein orbifolds (with isolated singularities).

\begin{defn}[Orbifold (with isolated singularities)]\label{orb Ein}
    We will say that a metric space $(M_o,g_o)$ is an orbifold of dimension $n\in\mathbb{N}$ if there exists $\epsilon_0>0$ and a finite number of points $(p_k)_k$ of $M_o$ called \emph{singular} such that we have the following properties:
    \begin{enumerate}
        \item the space $(M_o\backslash\{p_k\}_k,g_o)$ is a manifold of dimension $n$,
        \item for each singular point $p_k$ of $M_o$, there exists a neighborhood of $p_k$, $ U_k\subset M_o$, a finite subgroup acting freely on $\mathbb{S}^{n-1}$, $\Gamma_k\subset SO(n)$, and a diffeomorphism $ \Phi_k: B_e(0,\epsilon_0)\subset\mathbb{R}^n\slash\Gamma_k \to U_k\subset M_o $ for which, the pull-back of $\Phi_k^*g_o$  on the covering $\mathbb{R}^n$ is smooth.
    \end{enumerate}
\end{defn}

\begin{rem}
    Orbifold metrics will most of the time be indexed by a little $o$ for \emph{orbifold}.
\end{rem}

Theorem \ref{dégénérescence dim 4} only describes the formation of singularities at the scale of the maximum of the curvature where a Ricci-flat ALE \emph{manifold} appears as a blow up, but there might actually be other singularities forming at different scales. They are modeled on Ricci-flat ALE \emph{orbifolds}. The relevant model spaces for us are therefore Einstein orbifolds and Ricci-flat ALE orbifolds with isolated singularities.

\begin{defn}[ALE orbifold (with isolated singularities)]\label{def orb ale}
    An ALE orbifold of dimension $n\in\mathbb{N}$, $(N,g_{b})$ is a metric space for which there exists $\epsilon_0>0$, singular points $(p_k)_k$ and a compact $K\subset N$ for which we have:
    \begin{enumerate}
        \item $(N,g_b)$ is a orbifold of dimension $n$,
        \item there exists a compact subset $K\subset N$ and a diffeomorphism $\Psi_\infty: (\mathbb{R}^n\slash\Gamma_\infty)\backslash B_e(0,\epsilon_0^{-1}) \to N\backslash K$ such that we have $$r_e^l|\nabla^l(\Psi_\infty^* g_b - g_e)|_{C^2(g_e)}\leqslant C_l r_e^{-n}.$$
    \end{enumerate}
\end{defn}
\begin{rem}
    ALE metrics will most of the time be indexed by a little $b$ for \emph{bubble}.
\end{rem}

\subsection{Naïve desingularizations of an orbifold by trees of singularities}

By \cite{ban}, under the assumptions of Theorem \ref{dégénérescence dim 4}, at each singular point, a tree of Ricci-flat orbifolds forms. We index them by $b_j$ for $j$-th bubble.

\begin{defn}[Tree of singularities or desingularization pattern]
    Consider $(M_o,g_o)$ an Einstein orbifold and $S_o$ a subset of its singular points, and $(N_j,g_{b_j})_j$ a family of Ricci-flat ALE spaces asymptotic at infinity to $\mathbb{R}^4\slash\Gamma_j$ and $(S_{b_j})_j$ a subset of their singular points. Let us also assume that there is a one to one "gluing" map  $p:j\mapsto p_j\in S_o\cup \bigcup_k S_{b_k}$, where the singularity at $p_j$ is $\mathbb{R}^4\slash\Gamma_j$. We will call $D:= \big((M_o,g_o,S_o),(N_j,g_{b_j},S_{b_j})_j, p\big)$ a \emph{tree of singularities} or a \emph{desingularization pattern} for $(M_o,g_o)$ depending on the point of view.
\end{defn}

Let us define metrics which will mimic the degeneration of Einstein manifolds of Theorem \ref{dégénérescence dim 4}, they will be called \emph{naïve desingularizations} and the main goal of this paper is to show that degenerations of Einstein manifolds are well approximated by these naïve desingularizations.

Let us start by defining the naïve gluing of an ALE space $(N,g_b)$ to a singularity of an Einstein orbifold $(M_o,g_o)$. Recall the constant $\epsilon_0>0$ of Definitions \ref{orb Ein} and \ref{def orb ale}.

    Let $0<2\epsilon<\epsilon_0$, $0<t<\epsilon^4$, and $(M_o,g_o)$ be an orbifold and $\Phi: B_e(0,\epsilon_0)\subset\mathbb{R}^4\slash\Gamma \to U$ a local chart around a singular point $p\in M_o$ which satisfies $\Phi^*g_o = g_e + \mathcal{O}(r_e^2)$. Let also $(N,g_b)$ be an ALE orbifold asymptotic to $\mathbb{R}^4\slash\Gamma$, and $\Psi_\infty: (\mathbb{R}^4\slash\Gamma_\infty)\backslash B_e(0,\epsilon_0^{-1}) \to N\backslash K$ a chart at infinity in wich $\Psi_\infty^* g_b = g_e + \mathcal{O}(r_e^{-4})$. Consider finally $\chi$ a $C^\infty$-function from $\mathbb{R}^+$ to $\mathbb{R}^+$ supported on $[0,2]$ and equal to $1$ on $[0,1]$ and for all $s>0$, define $\phi_s: x\in \mathbb{R}^4\slash\Gamma\to sx \in \mathbb{R}^4\slash\Gamma$.
    
    We then define the orbifold $M_o\#N$ as $N$ glued to $M_o$ thanks to the diffeomorphism $ \Phi\circ\phi_{\sqrt{t}}\circ\Psi^{-1} : \Psi(A_e(\epsilon_0^{-1},\epsilon_0t^{-1/2}))\to \Phi(A_e(\epsilon_0^{-1}\sqrt{t},\epsilon_0))$ which identifies their annuli $\Phi(A_e(\epsilon_0^{-1}\sqrt{t},\epsilon_0))$ and $\Psi(A_e(\epsilon_0^{-1},\epsilon_0t^{-1/2}))$.

\begin{defn}[Naïve gluing of an ALE space to an orbifold singularity]
    We define a \emph{naïve desingularization} of an orbifold $(M_o,g_o)$ at a singular point $p$ by an ALE metric $(N,g_b)$ at scale $t$, which we will denote $(M_o\#N,g_o\#_{p,t}g_b)$ by putting $g_o\#_{p,t}g_b=g_o$ on $M_o\backslash U$, $g_o\#_{p,t}g_b=tg_b$ on $K\subset N$, and 
    $$g_o\#_{p,t}g_b = \chi(t^{-1/4}r_e)\Phi^*g_o + \big(1-\chi(t^{-1/4}r_e)\big) \phi_{t^{1/2},*}\Psi_\infty^*g_b$$
    on $A(t,\epsilon):=A_e(\epsilon^{-1}\sqrt{t},b\epsilon)$.
\end{defn}

More generally, it is possible to desingularize iteratively by trees of Ricci-flat ALE orbifolds. Let $0<2\epsilon<\epsilon_0$, $(M_o,g_o)$ be an Einstein orbifold and $(N_j,g_{b_j})_j$ be a tree of ALE Ricci-flat orbifolds such that $M = M_o\#_jN_j$ is the result of a desingularization pattern $D$, and let $0<t_j<\epsilon^4$ be \emph{relative gluing scales}.

\begin{defn}[Naïve desingularization of an orbifold by a tree or Ricci-flat orbifolds] \label{def desing naive}

    The \emph{naïve desingularization} metric $g^D_t$ on $M = M_o\#_jN_j$ is then the result of the following iterative construction. 
    
    Start with a deepest bubble $(N_j,g_{b_j})$, that is, $j$ such that $S_j= \emptyset$. If $p_j\in N_k$ and we replace $(N_k,g_{b_k},S_j)$ and $(N_j,g_{b_j},\emptyset)$ by $(N_k\#N_j,g_{b_k}\#_{p_j,t_j}g_{b_j},S_k\backslash\{p_j\})$ and restrict $p$ as $l\to p_l$ for $l\neq j$ in $D$ and consider another deepest bubble, the same works if $p_j\in M_o$.
\end{defn}

    Moreover, if $N_j$ is glued to $p_j\in N_{j_1}$, and $N_{j_1}$ is glued to $p_{j_1}\in N_{j_2}$, ..., $N_{j_{k-1}}$ is glued to $N_{j_k}$, which is glued to $M_o$, we define $T_j:= t_{j_1}t_{j_2}...t_{j_k}$. This way, on each $N_j(2\epsilon)$, the metric is $T_jg_{b_j}$.

\begin{rem}
    Our construction depends on a gauge choice: the diffeomorphisms used to glue the infinity of our ALE spaces to orbifold singularities can be composed with any isometry of $\mathbb{R}^4\slash \Gamma_k$. There is therefore a gluing gauge $\phi_k\in Isom(\mathbb{R}^4\slash\Gamma_k)$ at each point. These choices are equivalent to gluing different ALE spaces and they lead to different metrics, hence we will not worry too much about this degree of freedom here.
\end{rem}

\subsection{Coordinates on a naïve desingularization}

The above metric $g^D_t$ on $M$ of Definition \ref{def desing naive} has some well identified regions in which the metric is that of the orbifold, of one of the Ricci-flat ALE spaces or close to an annulus of a flat cone

\begin{defn}[Coordinates on the model spaces: $M_o(\epsilon)$, $N_j(\epsilon)$ and $A_k(t,\epsilon)$]
With the notations of definitions \ref{orb Ein} and \ref{def orb ale}, for $0<\epsilon\leqslant\epsilon_0$, we will denote
\begin{itemize}
    \item $M_o(\epsilon):= M_o\backslash  \Big(\bigcup_k \Phi_k(B_e(0,\epsilon)) \Big)\subset M_o,$
    \item $N_j(\epsilon):= N_j\backslash  \Big(\bigcup_k \Psi_k(B_e(0,\epsilon)) \cup \Psi_\infty \big((\mathbb{R}^4\slash\Gamma_\infty)\backslash B_e(0,\epsilon^{-1})\big)\Big)\subset N_j,$
    \item $A_j(t,\epsilon):=A_e(\epsilon^{-1}\sqrt{T_j}\sqrt{t_j},\epsilon\sqrt{T_j})\subset \mathbb{R}^4\slash\Gamma_k$.
\end{itemize}
    We see that these sets naturally embed in our manifold $M$ thanks to the above naïve desingularization construction. 
\end{defn}

\begin{rem}
    The sets $M_o(\epsilon)$ and $N_j(\epsilon)$ are the compact sets of the orbifolds minus their singular points which appear in \cite[Theorem A]{ban}. They are exhaustive as $\epsilon$ tends to zero.
\end{rem}

They moreover form a covering of $M = M_o\#_jN_j$ adapted to a degeneration at relative scales $(t_j)_j$:
\begin{equation}
    M = \bigcup_j\big(N_j(\epsilon)\cup A_j(t,\epsilon)\big) \cup   M_o(\epsilon).
\end{equation}

\subsection{Coordinates on an Einstein metric}

The goal will be to identify each of these regions with a region of an Einstein manifold $(M,g)$ close enough to an Einstein orbifold $(M_o,g_o)$ in the Gromov-Hausdorff sense.

\begin{defn}[Manifold $\epsilon$-approximated by a naïve desingularization]
Let $\epsilon>0$, $l\in\mathbb{N}$ and fix $t = (t_j)_j$ with $t_j<\epsilon^4$ for all $j$. We will say that a Riemannian manifold $(M,g)$ is \textit{$\epsilon$-approximated in  $C^l$-norm} by a naïve desingularization $(M,g^D_{t})$ of $(M_o,g_o)$, if there exists a diffeomorphism $\Phi: M\to M$ such that if we denote $\Phi(M_o(\epsilon)) = \mathcal{M}_o(\epsilon)\subset M$ and $\Phi(N_j(\epsilon)) = \mathcal{N}_{j}(\epsilon)\subset M$ and $\mathcal{A}_k(t,\epsilon)$ the region of $ M\backslash \mathcal{M}_o(32\epsilon)\cup \bigcup_j \mathcal{N}_j(32\epsilon)$ with nonempty intersection with $\Psi_\infty \big( B_e(0,\epsilon^{-1})\big)\subset\mathcal{N}_k(\epsilon)$:
	\begin{enumerate}
	\item on $\mathcal{M}_o(\epsilon)$ in $M$, we have
		$$\left\|\Phi^*g-g_o\right\|_{C^l(g_o)}=\left\|\Phi^*g-g^D\right\|_{C^l(g^D)}\leqslant \epsilon,$$
	\item on the zone $\mathcal{N}_{j}(\epsilon)$ in $M$, we have
	$$\left\|\frac{\Phi^*g}{T_{j}}-g_{b_j}\right\|_{C^l(g_{b_j})} = \left\|\Phi^*g-g^D\right\|_{C^l(g^D)} \leqslant \epsilon,$$
	\item and on $\mathcal{A}_k(t,\epsilon)$, we have
	$$\int_{\mathcal{A}_k(t,\epsilon)}|\Rm|^2dv_g<\epsilon^2.$$
	\end{enumerate}
\end{defn}
\begin{prop}\label{description dégénérescence}
Under the assumptions of Theorem \ref{dégénérescence dim 4}, given a sequence $(M_i,g_i)_i$ of Einstein manifolds converging in the Gromov-Hausdorff sense to an Einstein orbifold $(M_o,g_o)$, there exist a subsequence still denoted $(M_i,g_i)_i$ and naïve desingularizations $(M_i,g^D_{t_i})$ such that for all $\epsilon>0$ and $l\in \mathbb{N}$, for $i$ large enough, the manifold $(M_i,g_i)$ is $\epsilon$-approximated in $C^l$-norm by $(M_i,g^D_{t_i})$.
\end{prop}
\begin{proof}
Let $(M_i,g_i) \xrightarrow[GH]{} (M_o,g_o)$ be a sequence of manifolds satisfying the assumptions of Theorem \ref{dégénérescence dim 4}. Up to taking a subsequence, for any singular point of $M_o$, there exists a sequence of scales $T_i>0$ and of points $p_{i}$ of $ M_i^4$ of high curvature $|\textup{Rm}(g_i)|(p_{i})\rightarrow\infty$ such that: $\left(M^4_{i},\frac{g_{i}}{T_{i}},p_{i}\right)$ converging smoothly on compact subsets to $(N,g_b,p)$ a Ricci-flat ALE orbifold without its singular points. The convergence is moreover smooth at a bounded from below distance from the singular points since we are considering Einstein metrics with bounded curvature and harmonic radii bounded from below.

Let $\epsilon>0$ and $l\in \mathbb{N}$, by considering the compact $N_j(\epsilon)$, there exists $i_j\in \mathbb{N}$ such that for all $i>i_j$, there is a diffeomorphism on its image $\Phi_i: N(\epsilon)\to M_i$ with $\Big\|\frac{\Phi_i^*g_{i}}{T_{i}}-g_{b_j}\Big\|_{C^l} \leq \epsilon$ and similarly, there exists a rank $i_o$ from which we have the $C^l$-closeness on $M_o(\epsilon)$. Since there are only finitely many scales in each tree of singularities, we can take the maximum of the ranks $i_j$ and $i_o$ to obtain the result.

The control of the $L^2$-norm of the Riemannian curvature is proven in \cite{ban}.
\end{proof}

The main objective of this paper will be to construct good coordinates in the neck regions $\mathcal{A}_{k}(t,\epsilon)$. These regions $\mathcal{A}_{k}(t,\epsilon)$ between $N_k$ and $N_j$ are included in metric annuli $A\Big(\frac{1}{8}T_{k}^\frac{1}{2}\epsilon^{-1},8T_{j}^\frac{1}{2}\epsilon\Big)$ centered at $p_{k}$, and in $A\Big(\frac{1}{8}t_{k}^\frac{1}{2}\epsilon^{-1},8\epsilon\Big)$ if $N_k$ is directly glued to $M_o$ (in this case, $t_k = T_k$ because it is the shallowest scale). Let us note that the volume growth is almost Euclidean on them.

\begin{lem}\label{volume cone}
    For any $0<\delta<1$, there exists $\epsilon>0$ such that if a manifold $(M,g)$ is $\epsilon$-approximated by a naïve desingularization $g^D_t$ in $C^0$-norm for $t<\epsilon^4$ with $\Ric(g)\geqslant -3g$, then, if $\mathcal{A}_{k}(t,\epsilon)\subset A\Big(\frac{1}{8}T_{k}^\frac{1}{2}\epsilon^{-1},8T_{j}^\frac{1}{2}\epsilon\Big) = A(\rho_1,\rho_2)$, we have
    $$\Big|\frac{\vol_g(B_g(\delta^{-1}\rho_2))}{\vol_g(B_g(\delta\rho_1))}-\frac{(\delta^{-1}\rho_2)^4}{(\delta\rho_1)^4}\Big|\leqslant\delta.$$
\end{lem}
\begin{proof}
    Let us consider the case of an annular region $\mathcal{A}_k(t,\epsilon)$ between a Ricci-flat ALE $(N,g_b)$ and an orbifold $(M_o,g_o)$, which are asymptotic to $\mathbb{R}^4\slash\Gamma$, at scale $0<t<\epsilon^4$ for $\epsilon>0$ which we will choose small enough. By definition of an $\epsilon$-approximated metric, the ball $(B(\delta^{-1}\rho_2),g)$ is arbitrarily close in the Gromov-Hausdorff sense to the ball of radius $\delta^{-1}\epsilon$ of $(M_o,g_o) $ for $\epsilon$ arbitrarily small, and by definition of an orbifold metric, as $\delta^{-1}\epsilon\to 0$, we have $$\frac{\vol_{g_o}(B_{g_o}(\delta^{-1}\epsilon))}{(\delta^{-1}\epsilon)^4}\to \frac{\omega_4}{|\Gamma|},$$
    the volume of the unit ball of $\mathbb{R}^4\slash\Gamma$. 
    
    Similarly, the ball $(B_{g/t}(\delta\epsilon^{-1}),g/t)$ is arbitrarily close in the Gromov-Hausdorff sense to the ball of radius $\delta\epsilon^{-1}$ of $(N,g_b)$ whose volume satisfies for $\delta^{-1}\epsilon\to 0$
    $$\frac{\vol_{g_b}(B_{g_b}(\delta\epsilon^{-1}))}{(\delta\epsilon^{-1})^4}\to \frac{\omega_4}{|\Gamma|}.$$
    As a consequence, by the continuity of volume for the Gromov-Hausdorff distance between manifolds with $\Ric\geqslant- 3g$, \cite{col}, we have the stated result for $\epsilon$ small enough.
\end{proof}

The goal is now to construct a diffeomorphism from a flat cone of $\big(\mathbb{R}^n\slash\Gamma,g_e\big)$ to our intermediate annulus. Thanks to the controls on the curvature of \cite[Proposition 3]{ban}, we expect that a "natural" coordinate system must also have good enough controls. Here we choose to foliate our annuli by constant mean curvature hypersurfaces before constructing coordinates relying on them.

	\section{First coordinates}
The main difficulty here in constructing our coordinates when compared to the context of \cite{bkn} is that we do not have an asymptotic behavior of our metric on which we can base our construction. We will need to start by using a first set of local coordinates which is only partially satisfactory. The main goal of this section is the proof of the following proposition.

\begin{prop}\label{Premières coordonnées}
		For all $\delta >0$, $D_0>0$, $v_0>0$, $E>0$ and $l\in\mathbb{N}$, there exists $\epsilon_2>0$ such that if an Einstein manifold satisfies:
	\begin{itemize}
	\item its diameter is bounded by $D_0>0$,
	\item its volume is bounded below by $v_0>0$,
	\item its Ricci curvature is bounded below by $-3$,
	\item there exists an annulus $A(\rho_1,\rho_2)$ with $4\rho_1<\rho_2$ satisfying: 
	$$\int_{A(\rho_1,\rho_2)}|\Rm|^{2} \leqslant \epsilon_2^2,$$
	and $\int_{B(\rho_2)}|\Rm|^{2} \leqslant E$,
	\item the volume growth in the annulus is almost Euclidean $$\Big|\frac{\vol(B(\rho_2))}{\vol(B(\epsilon_2\rho_1))}-\frac{\rho_2^4}{(\epsilon_2\rho_1)^4}\Big|\leqslant\epsilon_2. $$
	\end{itemize}
	Then, for any $\rho\in \big[2\rho_1,\frac{1}{4}\rho_2\big]$ there exists a region $\hat{A}\big(\rho,2\rho\big)$ close to the annulus $A\big(\rho,2\rho\big)$ in the following sense,
	$$A\big((1+\delta)\rho,(2-\delta)\rho\big) \subset \hat{A}\big(\rho,2\rho\big)\subset A\big((1-\delta)\rho,(2+\delta)\rho\big),$$ 
	such that there exists $\Gamma$ a finite subgroup of $SO(4)$ and a diffeomorphism $$\Phi_\rho : A_{e}(1,2)\subset \mathbb{R}^4\slash\Gamma\to \hat{A}\big(\rho,2\rho\big)\subset M,$$
	for which we have :
	$$\Big\|\frac{\Phi_\rho^*g}{\rho^2}-g_e\Big\|_{C^{l}(A_{e}(1,2))}\leqslant \delta.$$
	\end{prop}

\subsection{$\epsilon$-regularity in the neck regions}

\begin{defn}[Hölder spaces]\label{norme holder}
	The $C^{k,\alpha}$-norms will be taken at the scale of the injectivity radius. Which means that we define the $C^{k,\alpha}$-norm of a tensor $s$ in the following way:
	$$\|s\|_{C^{k,\alpha}}:= \|s\|_{C^{0}}+...+\|\nabla_g^ks\|_{C^{0}} + \sup_x[\nabla_g^ks]_{C^\alpha(g)}(x),$$
	where, for a tensor $u$, a point $x$, $\alpha>0$ and a metric $g$, if we denote $\exp_x$ the exponential map at $x$ whose injectivity radius is $\textup{inj}_g(x)$, we define the Hölder seminorm of $u$ on $M$ as
    $$ [u]_{C^\alpha(g)}(x):= \sup_{\{y\in T_xM,|y|< \textup{inj}_g(x)\}} \Big| \frac{\exp_x^*u(0)-\exp_x^*u(y)}{|y|^\alpha} \Big|_{\exp_x^*g}.$$
\end{defn}

We will also sometimes use the \emph{harmonic radius} which is defined at $x\in M$, as the supremum of the $r>0$ for which there exists a diffeomorphism $\Phi: B(x,r)\to \mathbb{R}^n$, $\Phi(y)=(h_1(y),...,h_n(y))$ such that for all $i\in\{1,...,n\}$, $$\Delta_g h_i =0,$$
and denoting $g_{ij}$ the push forward of $g$ by $\Phi$,
$$ \|g_{ij}-\delta_{ij}\|_{C^0}+ r^{1+\alpha} \|g_{ij}\|_{C^{1+\alpha}}\leqslant \frac{1}{100}.$$

\begin{lem}\label{contrôle du rayon d'injectivité}
	Under the assumptions of Proposition \ref{Premières coordonnées}, there exists a constant $r_0>0$ only depending on $D_0$ and $v_0$, such that for all points of $A\big(2\rho_1,\frac{1}{2}\rho_2\big)$ at distance $2\rho_1<\rho<\frac{1}{2}\rho_2$ from the center, the injectivity radius and the harmonic radius are bounded below by $r_0 \rho$.
\end{lem}
\begin{proof}
Let us consider a point $x$ at distance $\rho\in [2\rho_1,\frac{1}{2}\rho_2]$ from the center of the annulus and the ball $B(x,\frac{\rho}{2})$. Thanks to Bishop-Gromov inequality, the volume of $B(x,\frac{\rho}{2})$ is bounded from below by $C_1\rho^4$ for a constant $C_1>0$ only depending on $v_0$ and $D_0$. To control the curvature, we use the following $\epsilon$-regularity theorem:
\begin{lem}[{\cite[Theorems 4.8, 4.9 and 4.11]{gao}}]
    In an Einstein manifold $(M,g)$ satisfying the assumptions of Proposition \ref{Premières coordonnées}, for any $C_0>0$ there exists a constant $\epsilon=\epsilon(C_0,v_0,D_0)$ such that if at a point $x\in M$ and for $\rho>0$ we have $\vol_g(B_g(x,\rho))>C_0 \rho^4$ and
\begin{equation}
\int_{B(x,2\rho)}|\Rm(g)|^2 dv(g) < \epsilon^2,\label{condition L2}
\end{equation} 
then, we have
\begin{equation}
|\Rm(g)|\leqslant \frac{1}{\rho^2} \label{conclusion Linfty}
\end{equation} on $B(x,\rho)$.
\end{lem}

Therefore for $\rho\in [4\rho_1, \frac{1}{4} \rho_2]$, up to shrinking the constant $\epsilon_2$ of the statement to be smaller than $\epsilon(16C_1)>0$, the ball $B(x,\frac{\rho}{2})$ satisfies \eqref{condition L2}. Hence, the curvature at distance $\rho$ is smaller than $\frac{1}{\rho^2}$ by \eqref{conclusion Linfty}.

By \cite[Theorem 4.3]{cgt}, these controls on the curvature and the volume at scale $\rho$ imply that there exists a constant $r_1>0$ depending on the volume of the ball of radius $\rho$ and on a bound on the curvature at this scale such that at any point at distance $\rho$ from the center of the annulus, the injectivity radius is bounded below by $r_1\rho$. Since the curvature is bounded, there also exists $0<r_0<r_1$ depending on $r_1$ and the bound on the curvature such that the harmonic radius is bounded below by $r_0\rho$.
	\end{proof}

\subsection{Construction of the first coordinates}

	Let us start by constructing first coordinates which we will improve later.
	
	The crucial tool to prove this result is the almost volume cone theorem of Cheeger and Colding.
		\begin{lem}[{\cite[Theorem 4.85]{cc}}]\label{almost volume cone} 
		For all $\delta>0$, there exists $\kappa>0$ such that for any $r>0$, if $(M^n,g)$ satisfies, $\Ric(g) \geqslant -(n-1)\kappa r^{-2}$ and there exists a point $p\in M$ at which,
		$$\frac{\vol(B(p,\kappa r))}{v_{-\kappa r^{-2}}(\kappa r)}-\frac{\vol(B(p,2r))}{v_{-\kappa r^{-2}}(2r)}<\kappa,$$ 
		where $v_{-\kappa r^{-2}}(s)$ is the volume of the ball of radius $s$ in the simply connected space with sectional curvatures constant equal to $-\kappa r^{-2}$,
	    then, there exists a metric cone $(C(X),d_{C(X)},0)$ such that $diam (X)\leqslant \pi$ and,
		$$d_{GH}\Big(\big(B(p,r),p\big),\big(B_{C(X)}(0,r),0\big)\Big)\leq \delta r.$$
		\end{lem}

		We can refine this result in the case of $4$-dimensional Einstein manifolds.
		\begin{lem}[Almost volume cone in a $4$-dimensional Einstein manifold]\label{presque cone variété E dim 4}
		For all $\delta>0$, $l\in \mathbb{N}$, $v_0>0$, and $E>0$, there exists $\kappa>0$ such that for all $r>0$, if $(M^4,g)$ satisfies $\Ric(g) = \Lambda g$ with, $ |\Lambda|\leqslant 3\kappa r^{-2}$ and if there exists a point $p\in M$ for which, 
		$$\vol(B(p,r)) \geqslant v_0 r^4,$$
		$$\frac{\vol(B(p,\kappa r))}{v_{-\kappa r^{-2}}(\kappa r)}-\frac{\vol(B(2r))}{v_{-\kappa r^{-2}}(2r)}<\kappa,$$
		and $$\int_{B(p,2r)} |\Rm|^2dv \leqslant E,$$
		then, there exists a finite subgroup $\Gamma\subset SO(4)$ (with order bounded in terms of $v_0$), and a diffeomorphism
		$\Phi:A_e(\delta , (1-\delta )) \to \hat{A}(\delta r, (1-\delta )r) $, where,
		  \begin{equation}
		 A(2\delta r, (1-2\delta )r) \subset \hat{A}(\delta r, (1-\delta )r) \subset A\Big(\frac{1}{2}\delta r, \Big(1-\frac{1}{2}\delta \Big)r\Big), \label{arrivée}
		 \end{equation}
		 such that on the annulus $A_e(\delta r, (1-\delta )r)$, we have:
		 \begin{equation}
		 \Big\|\frac{\Phi^*g}{r^2}-g_e\Big\|_{C^l(g_e)}\leqslant \delta .\label{controle diffeo}
		 \end{equation}
		\end{lem}
\begin{rem}
    The assumptions on the Riemannian curvature and the volume are there to bound the order of $\Gamma$. They prevent our Einstein manifolds to be too close to $\mathbb{R}^4\slash\Gamma_i$ where the order of $\Gamma_i$ goes to infinity. In such a situation, the sequence could converge to the cone $\mathbb{R}^+$. These hypotheses are moreover redundant by \cite[Theorem 1.13]{cn}, the hypotheses on the Ricci curvature and the volume imply a bound on the $L^2$-norm of the curvature.
\end{rem}		
		\begin{proof}
		By rescaling we can restrict our attention to $r =1$ and assume towards a contradiction that there exist $\delta>0$, $E>0$ and $v_0>0$, and a sequence of counterexamples $(M^4_i,g_i,p_i)$ satisfying \begin{itemize}
		    \item $\vol(B_i(p_i,1))>v_0>0$,
		    \item $\int_{B_i(p_i,2)} |\Rm_i|^2\textup{d}v_i \leqslant E$,
		    \item there exists $\kappa_i\to 0$ and $\Lambda_i$, $ |\Lambda_i|\leqslant 3\kappa_i$ for which $\Ric(g_i) = \Lambda_i g_i$ and $\frac{\vol(B_i(p_i,\kappa_i ))}{v_{-\kappa_i }(\kappa_i )}-\frac{\vol(B_i(p_i,2))}{v_{-\kappa_i }(2)}<\kappa_i$,
		\end{itemize} 
		but such that for any $\Gamma$ finite subgroup of $SO(4)$, there does not exist a diffeomorphism $\Phi:A_e(\delta , (1-\delta )) \to \hat{A}(\delta r, (1-\delta )r)$ where $\hat{A}(\delta r, (1-\delta )r)$ satisfies \eqref{arrivée}, and such that we have the control \eqref{controle diffeo}.
		
		Since $\vol(B_i(p_i,1))>v_0>0$ and $\Ric(g_i)\geqslant -3g_i$, a subsequence of the balls $B_i(p_i,1)$ converges to a metric space $(Y,d_Y,p)$ in the Gromov-Hausdorff sense. By Lemma \ref{almost volume cone}, up to taking a subsequence, for all $i$, we can assume that $\kappa_i$ is small enough to have $$d_{GH}\Big(\big(B_i(p,1),p_i\big),\big(B_{C(X_i)}(0,1),0\big)\Big)\leqslant \frac{1}{i},$$ for a particular metric cone $(C(X_i),d_{C(X_i)},0)$ with $diam (X_i)\leqslant \pi$. By the triangular inequality, the sequence $\big(B_{C(X_i)}(0,1),0\big)$ also converges to $(Y,d_Y,p)$ in the Gromov-Hausdorff sense. Now, since the $C(X_i)$ with the distances $d_{C(X_i)}\big((r,x),(s,y)\big):= \sqrt{r^2+s^2-2rs \cos(d_{X_i}(x,y))}$ converge, we deduce by fixing $r$ and $s$ that the $(X_i,d_{X_i})$ converge to $(X,d_{X})$, and therefore that the limit is a metric cone, $ (Y,d_Y,p)=(C(X),d_{C(X)},0)$.

        By the compactness theorem for noncollapsed Einstein metrics with bounded $L^2$-norm for the curvature \cite{bkn,and}, $(Y,d_Y,p)$ is an Einstein orbifold. The limit of $(B_i(p_i,1),g_i,p_i)$ in the Gromov-Hausdorff sense is therefore both the radius $1$ ball of a metric cone and of an Einstein orbifold of dimension $4$. Since an Einstein orbifold has bounded curvature, the cone has to be flat. The limit is therefore $\mathbb{R}^4\slash\Gamma_\infty$ for $\Gamma_\infty$ a finite subgroup of $SO(4)$. 
		
		Now, the convergence of Einstein metrics is smooth on compact subsets of the complement of the singularities by \cite{bkn}, and there is only one here at the singular point of $\mathbb{R}^4\slash\Gamma_\infty$, hence, there exists a diffeomorphism $\Phi:A_e(\delta , (1-\delta )) \to \hat{A}(\delta r, (1-\delta )r)$ satisfying \eqref{arrivée} and \eqref{controle diffeo}, which is a contradiction.
		\end{proof}

We can now use these results to construct some coordinates in our annulus.

\begin{proof}[Proof of Proposition \ref{Premières coordonnées}]
	Let $\delta >0$, $D_0>0$, $v_0>0$, $E>0$, $\Lambda \in [-3,3]$ and $l\in\mathbb{N}$, and a $4$-dimensional manifold, $(M,g)$ satisfiying :
	\begin{itemize}
	\item $\Ric(g) = \Lambda g$,
	\item the diameter is bounded by $D_0>0$,
	\item the volume is bounded below by $v_0>0$,
	\item there exists $p\in M$, and $\rho_2>0$, such that we have : 
	$\int_{B(p,\rho_2)}|\Rm|^{2} \leqslant E$.
	\end{itemize}
 Let us first remark that these imply a control on $\Lambda\rho_2^2$ by Bishop-Gromov inequality
\begin{align}
E&\geqslant \int_{B(\rho_2)}|\Rm|^2\textup{d}v\nonumber\\
&\geqslant 6\int_{B(\rho_2)}|\Ric|^2\textup{d}v\nonumber\\
&\geqslant 24\Lambda^2\vol(B(\rho_2))\nonumber\\
&\geqslant 24v_0\Lambda^2\frac{\vol_{g_{-|\Lambda|}}(B_{g_{-|\Lambda|}}(\rho_2))}{\vol_{g_{-|\Lambda|}}(B_{g_{-|\Lambda|}}(D_0))}\nonumber\\
&\geqslant 24v_0\Lambda^2\frac{\vol_{g_{-3}}(B_{g_{-3}}(\rho_2))}{\vol_{g_{-3}}(B_{g_{-3}}(D_0))}\nonumber\\
&\geqslant C(D_0,v_0)\Lambda^2\rho_2^4,\label{controle Lambda par E}
\end{align}
hence $\Lambda\rho^2_2\leqslant \frac{E^{1/2}}{C(D_0,v_0)^{1/2}} =: C'(E,D_0,v_0)$. 

Let $l\in \mathbb{N}$, and denote $0<\kappa_0 <\frac{1}{4}$ the constant of Lemma \ref{presque cone variété E dim 4} associated to our constants $\delta$, $l$ et $E$. We are looking for $\epsilon_2>0$ small enough to satisfy the hypotheses of Lemma \ref{presque cone variété E dim 4} with constant $\kappa_0>0$. 

For any $0<4\rho_1\leq \rho_2$ for which we have :
$$\int_{A(\rho_1,\rho_2)}|\Rm|^{2} \leqslant \epsilon_2^2,$$
and
$$\Big|\frac{\vol(B(\rho_2))}{\vol(B(\epsilon_2\rho_1))}-\frac{\rho_2^4}{(\epsilon_2\rho_1)^4}\Big|\leqslant\epsilon_2, $$
we actually have, thanks to Bishop-Gromov inequality, for
$C''(E,D_0,v_0)>0$,
\begin{equation}
    \Big|\frac{\vol(B(\rho_2))}{\vol(B(\rho_1))}-\frac{\rho_2^4}{\rho_1^4}\Big|\leqslant C''(E,D_0,v_0)\epsilon_2,\label{croissance volume rho1 rho2}
\end{equation}
and for any $\rho$ and $\kappa_0>0$ for which $ \rho_1<\kappa_0\rho<2\rho<\rho_2 $,
\begin{equation}
    \Big|\frac{\vol(B(2\rho))}{\vol(B(\kappa_0\rho))}-\frac{(2\rho)^4}{(\kappa_0\rho)^4}\Big|\leqslant C''(E,D_0,v_0)\epsilon_2.\label{croissance volume rho}
\end{equation}
We can therefore assume that $\vol(A(\rho_1,\rho_2))\geqslant c(E,D_0,v_0)\rho^4_2$ for $4\rho_1<\rho_2$ and $c(E,D_0,v_0)>0$ by choosing $\epsilon_2$ small enough. Let us then look for a bound on $\Lambda$ in order to verify the assumptions of Lemma \ref{presque cone variété E dim 4}. Since $\Ric = \Lambda g$, we have
\begin{align*}
\epsilon_2^2&\geqslant \int_{A(\rho_1,\rho_2)}|\Rm|^2\textup{d}v\\
&\geqslant 6\int_{A(\rho_1,\rho_2)}|\Ric|^2\textup{d}v\\
&\geqslant 24\Lambda^2\vol(A(\rho_1,\rho_2))\\
&\geqslant 24c(E,D_0,v_0)\Lambda^2\rho^4_2
\end{align*}
which rewrites $|\Lambda| \leqslant  \epsilon_2(24c(E,D_0,v_0))^{-1/2} \rho_2^{-2}.$ Therefore, for $\epsilon_2$ small enough, we verify the assumptions of Lemma \ref{presque cone variété E dim 4} for any $\rho$ by choosing $\epsilon_2$ for which we have $$ \kappa_0 = \epsilon_2\max((24c(E,D_0,v_0))^{-1/2},1,C''(E,D_0,v_0)), $$ because we control the volume growth between $\kappa_0 \rho$ and $2\rho$ thanks to \eqref{croissance volume rho}. The conclusion of Proposition \ref{Premières coordonnées} therefore holds for any $\rho_1<\rho<\frac{1}{4}\rho_2$ with $\delta(\kappa_0)>0$.
	\end{proof}

	\subsection{Consequences of the existence of the first coordinates.}
	    The diffeomorphism of Proposition \ref{Premières coordonnées} allows us to pull back our Einstein metric on a flat annulus $A_{e}(\rho_1, \rho_2) \subset \mathbb{R}^4\slash \Gamma$, where the Einstein metric $g$ is $C^\infty$-close to the flat metric $g_e$. This will let us construct hypersurfaces of almost constant second fundamental form for our Einstein metric.
		
	    \begin{rem}
	    Here and in the following, we will use the convention that our normal vectors point outwards, and we define the second fundamental form as $$A^\Sigma(X,Y):= \langle X,\nabla_{Y}N^\Sigma\rangle,$$ which ensures that spheres in the Euclidean space have positive second fundamental form.
	    \end{rem}
		
			\begin{lem}\label{hypersurfaces presque sphériques}
			For any $l\in\mathbb{N}$, there exists $C>0$ and $\delta_0>0$ such that if a metric $g$ on $A_e\big(\frac{1}{4},4\big)$ satisfies,
				$$\|g-g_e\|_{C^{l+1}\big(A_e\big(\frac{1}{4},4\big)\big)}\leqslant \delta_0,$$
		we have, for all $s\in \big[\frac{1}{4},4\big]$, the sphere of radius $s$ of $\mathbb{R}^n\slash \Gamma$, $\mathcal{S}_s$, has a second fundamental form for $g$, $A_g^{\mathcal{S}_s}$, which satisfies
		$$\Big\|\frac{A_g^{\mathcal{S}_s }}{s}-\frac{g^{\mathcal{S}_s}}{s^2}\Big\|_{C^{l}\big(\frac{g^{\mathcal{S}_s }}{s^2}\big)}\leqslant C\|g-g_e\|_{C^{l+1}\big(A_e\big(\frac{1}{4},4\big)\big)}.$$
		We moreover have
		$$\Big\|\frac{g^{\mathcal{S}_s }}{s^2}-g^{\mathbb{S}^{n-1}\slash \Gamma}\Big\|_{C^{l+1}(g^{\mathbb{S}^{n-1}\slash \Gamma})} \leqslant C\|g-g_e\|_{C^{l+1}\big(A_e\big(\frac{1}{4},4\big)\big)}.$$
		\end{lem}
		\begin{proof}
		The second fundamental form involves one derivative of the metric which is controlled in $C^{l+1}$-norm by assumption. Since for $g_e$, the hypersurfaces $\mathcal{S}_s$ have a second fundamental form constant equal to $\frac{g^{\mathbb{S}^{n-1}\slash \Gamma}}{s^2}$, we conclude that we indeed control their second fundamental form $$\Big\|\frac{A_g^{\mathcal{S}_s }}{s}-\frac{g^{\mathcal{S}_s}}{s^2}\Big\|_{C^{l}\big(\frac{g^{\mathcal{S}_s }}{s^2}\big)}\leqslant C\|g-g_e\|_{C^{l+1}\big(A_e\big(\frac{1}{4},4\big)\big)}.$$ In a similar fashion, we get the control on the metric.
		\end{proof}
		
	Given a smooth and orientable hypersurface, $\Sigma$, we define for $x\in \Sigma$ and $s \in \mathbb{R}$, $F^g(x,s):=\gamma^g_x(s)$, where $\gamma^g_x$ is the geodesic for $g$ such that $\gamma^g_x(0) = x$, and ${\gamma'}_x^g(0)$ is the normal to $\Sigma$ at $x$. We define the \emph{normal injectivity radius of $\Sigma$}, denoted $\textup{inj}_g^\perp(\Sigma)$ as the supremum of the constants $r>0$ for which the map $F^g: \Sigma\times [-r,r]\to M$ is a diffeomorphism on its image.
	
		\begin{lem}\label{injectivité normale}
		For all $l\in\mathbb{N}$ and  $\epsilon>0$, there exists $C>0$ and $\delta>0$ such that if a metric $g$ satisfies $\|g-g_e\|_{C^{l+2}\big(A_e\big(\frac{1}{4},4\big)\big)}\leqslant \delta,$
		then for any hypersurface $\Sigma \subset A_e\big(\frac{1}{2},2\big) $ satisfying $\|A^{\Sigma}_{g_e}\|_{C^{l+1}}\leqslant 10$, we have the two following properties:
 \begin{itemize}
 		\item the normal injectivity radii of $\Sigma$ for $g$ $\textup{inj}_g^\perp(\Sigma)$, and $g_e$,   $\textup{inj}_{g_e}^\perp(\Sigma)$ satisfy
 $$\Big|\frac{\textup{inj}_g^\perp(\Sigma)}{\textup{inj}_{g_e}^\perp(\Sigma)} - 1\Big| \leqslant \epsilon.$$
 
		\item the two normal exponential maps are also close to each other on the set $$\Sigma\times \big[-(1-\epsilon)\textup{inj}_{g_e}^\perp(\Sigma),(1-\epsilon)\textup{inj}_{g_e}^\perp(\Sigma)\big]:$$ there exists a vector field $X$ on $\mathbb{R}^n\slash\Gamma$, $\|X\|_{C^{l+1}(g_e)}<\epsilon$,  for which we have,
			$$F^{g}  \circ  (F^{g_e})^{-1} = \phi_X,$$
			where $\phi_X:p\mapsto \exp^{g_e}_{p}(X(p))$.
			\end{itemize}
	\end{lem}
	\begin{proof}
		The curves $F^{g}(x,.) = \gamma^{g}_x(.)$ and $F^{g_e}(x,.) = \gamma^{g_e}_x(.)$ are respectively the solutions of the geodesic equation for $g$ and $g_e$,
		\begin{equation}
  \left\{
      \begin{aligned}
        \nabla_{\gamma_x'} \gamma_x' &\;= 0,\\
        \gamma_x(0) &\;= x,\\
        \gamma_x'(0) &\;= N_x.
      \end{aligned}
    \right.
\end{equation}
where $N^{g}_x$ and $N^{g_e}_x$ are the outwards normal to $\Sigma$ for the two metrics. By assumption, the metrics $g$ and $g_e$ are $C^{l+2}$-close. This implies that $N_x^g$ and $N^{g_e}_x$ are $C^{l+2}$-close, and that $\nabla^{g}$ and $\nabla^{g_e}$ are $C^{l+1}$-close.

By continuity with respect to the initial conditions and parameters of order $2$ differential equations, given a fixed hypersurface $\Sigma$, for all $\epsilon>0$, there exists $\delta>0$ such that if $\|g-g_e\|_{C^{l+2}}\leqslant\delta$, then, the difference between the normal geodesics is small in the following $C^l$-sense since we work in local charts. There exists a vector field $X$ defined along each $\gamma^{g_e}_x$ such that $$\gamma^{g}_x(s) = \phi_X(\gamma^{g_e}_x(s)):= \exp_{\gamma^{g_e}_x(s)}\big(X(\gamma^{g_e}_x(s))\big),$$
and $\|X\|_{C^{l+1}(g_e)}<\epsilon$. Hence, the equality $\gamma^{g}_x(s) = \phi_X(\gamma^{g_e}_x(s))$ rewrites:
$$F^{g} = \phi_X\circ F^{g_e}.$$
	\end{proof}
	
	In particular, we deduce that if a hypersurface $\Sigma'$ is a small normal perturbation of $\Sigma$ for the metric $g$, then, it is also a small normal perturbation of $\Sigma$ for the metric $g_e$, and conversely.
	
			\begin{cor}\label{pert normale de chaque côté}
		For all $\epsilon_0$ $l\in\mathbb{N}$, there exists $C>0$ and $\delta_0>0$, such that for any metric $g$ satisfying,
				$$\|g-g_e\|_{C^{l+2}\big(A_e\big(\frac{1}{4},4\big)\big)}\leqslant \delta_0,$$
		and for any hypersurface $\Sigma \subset A_e\big(\frac{1}{2},2\big) $ satisfying $\|A^{\Sigma}_{g_e}\|_{C^{l+1}}\leqslant 10$,
			we have: for any function $w:\Sigma\to \mathbb{R}$ with $\|w\|_{C^l}\leqslant \epsilon_0$, and denoting $\Sigma(w):= \{F(x,w(x)), x\in \Sigma\}$,
			\begin{enumerate}
			\item there exists $w':\Sigma\to \mathbb{R}$ with $\|w'\|_{C^l}\leqslant 2\epsilon_0$ such that
			$$\Sigma_{g_e}(w) = \Sigma_g(w'),$$
			\item there exists $w'':\Sigma\to \mathbb{R}$ with $\|w''\|_{C^l}\leqslant 2\epsilon_0$ such that
			$$\Sigma_g(w) = \Sigma_{g_e}(w'').$$
			\end{enumerate}
		\end{cor}
		\begin{proof}
		According to Lemma \ref{injectivité normale}, we have $F^{g} = \phi_X\circ F^{g_e}$, where $X$ is arbitrarily small in $C^{l+1}$-norm for $\delta_0$ arbitrarily small. Adding the assumption $\|A^{\Sigma}_{g_e}\|_{C^{l+1}}\leqslant 10$, ensures that for $\epsilon_0>0$, the metrics $g_{e}:=(F^{g_e})^*g_e$ and $g_{p}:=(g_e)_{|\Sigma} + ds^2$ on $\Sigma\times [-2\epsilon_0,2\epsilon_0]$,  are arbitrarily close in $C^{l}$-norm for $\epsilon_0$ arbitrarily small. This implies that
		$$\Big[ (F^{g_e})^{-1} \circ \phi_X\circ F^{g_e} \Big] = \phi^{e}_{X^{e}},$$
		where $ X^e $ is the pull-back of $X$ by $F^{g_e}$, and $\phi^{e}_{X^{e}}(p):= \exp^{g_{e}}_{p}(X^{e}(p))$. Hence, we have $\|X^e\|_{C^{l+1}(g_e)}\leqslant \epsilon_0$ by the proximity of the metrics $g_e$ and $g_p$, $ \phi^{e}_{X^{e}} = (\phi^\Sigma_{X_\Sigma^{p}},\phi^\mathbb{R}_{X_\mathbb{R}^p}) $, where $\|X_\Sigma^{p}\|_{C^l(g^\Sigma)}+ \|X_\mathbb{R}^{p}\|_{C^l(ds^2)} \leqslant 2\epsilon_0$.
		
	    Now, for all $(x,s) \in \Sigma\times [-2\epsilon_0,2\epsilon_0]$, we have
			\begin{align*}
				F^{g} (x,s) &\;= F^{g_e} \circ \Big[ (F^{g_e})^{-1} \circ \phi_X\circ F^{g_e} \Big](x,s),\\
				&\;= F^{g_e}  \circ (\phi^\Sigma_{X_\Sigma^{p}},\phi^\mathbb{R}_{X_\mathbb{R}^p})(x,s)\\
				&\;= F^{g_e} \big(\phi^\Sigma_{X_\Sigma^{p}}(x),\phi^\mathbb{R}_{X_\mathbb{R}^p}(s)\big).
			\end{align*}
			In particular, for all $w: \Sigma \to [-2\epsilon_0,2\epsilon_0]$, we have $$F^{g} (x,w(x)) = \big(\phi^\Sigma_{X_\Sigma^{p}}(x),w'( \phi^\Sigma_{X_\Sigma^{p}}(x))\big),$$
			where
			$w':=\phi^\mathbb{R}_{X_\mathbb{R}^p}\circ w\circ \big(\phi^\Sigma_{X_\Sigma^{p}}\big)^{-1}$. By the controls of $X_\mathbb{R}^p$ and $X_\Sigma^{p}$ which we can assume as small as wanted by choosing $\delta_0$ and $\epsilon_0$ small enough, we get the stated result.
		\end{proof}

\section{Controls on constant mean curvature hypersurfaces in low curvature regions}

Let us motivate the use of constant mean curvature hypersurfaces. We control both the way they are embedded in our manifold and their intrinsic geometry thanks to the ambient curvature which we will control in our neck regions by \cite[Proposition 3]{ban}.

\begin{rem}
All along this section, we will denote constant mean curvature hypersurfaces with a tilde to insist on their particularity.
\end{rem}

\begin{prop}\label{contrôle CMC}
For all $k\in \mathbb{N}$ and $0<\alpha<1$, there exists $\epsilon>0$ and $C>0$ for which, if $\tilde{\Sigma}$ is a hypersurface of a manifold $(M^n,g)$ satisfying:
\begin{enumerate}
\item its mean curvature is constant equal to $n-1$, 
\item there exists $\Gamma$ a finite subgroup of $SO(n)$, and a diffeomorphism $\Phi: \mathbb{S}^{n-1}\slash\Gamma\to \tilde{\Sigma} $ such that we have
	$$\left\|\Phi^*g^{\tilde{\Sigma}} - g^{\mathbb{S}^{n-1}\slash\Gamma}\right\|_{C^{1,\alpha}(\mathbb{S}^{n-1}\slash\Gamma)} \leqslant \epsilon.$$
	\item the ambient curvature is small on $\tilde{\Sigma}$,
	$$\|\Rm(g)\|_{C^k(g)}\leqslant \epsilon,$$
	\end{enumerate}
	then:
	\begin{itemize}
		\item the second fundamental form is almost constant at the scale of curvature,
		$$\big\|A^{\tilde{\Sigma}} - g^{\tilde{\Sigma}}\big\|_{C^{k,\alpha}(g^{\tilde{\Sigma}})}\leqslant C\|\Rm(g)\|_{C^k(g)},$$	
	\item there exists a diffeomorphism $\Psi: \mathbb{S}^{n-1}\slash \Gamma \to \tilde{\Sigma}$ such that,
	$$\big\|\Psi^*g^{\tilde{\Sigma}} - g^{\mathbb{S}^{n-1}\slash\Gamma}\big\|_{C^{k+1,\alpha}(\mathbb{S}^{n-1}\slash\Gamma)} \leqslant C\|\Rm(g)\|_{C^k(g)}.$$
		\end{itemize}
\end{prop}

Let us precise the notations we will use for this proof:
	\begin{itemize}
		\item $\nabla^{\tilde{\Sigma}} $: the Levi-Civita connection of the hypersurface,
		\item $\nabla$: the Levi-Civita connection of the ambient space,
		\item $N^{\tilde{\Sigma}}$: the outwards normal to the hypersurface,
		\item $A^{\tilde{\Sigma}}(X,Y):= \langle X, \nabla_Y N^{\tilde{\Sigma}}\rangle$: the second fundamental form of the hypersurface which will be seen as a symmetric $2$-tensor of the tangent space to $\tilde{\Sigma}$,
		\item $H^{\tilde{\Sigma}}$: the mean curvature of the hypersurface,
		\item for $B$ a $1$-form on $\tilde{\Sigma}$ taking values in $T^*\tilde{\Sigma}$, we define 
				$$(d^{\nabla^{\tilde{\Sigma}}} B)(X,Y):= \nabla^{\tilde{\Sigma}}_XB(Y,.)-\nabla^{\tilde{\Sigma}}_YB(X,.),$$
		\item we will denote $\delta^{\nabla^{\tilde{\Sigma}}}$ the formal adjoint of $d^{\nabla^{\tilde{\Sigma}}}$.
	\end{itemize}
	
	We have the following classical formulas (see for instance \cite[Chapter 4]{pet}) for the second formula for any hypersurface $\Sigma$: for all vector fields $X,Y,Z$ on $\Sigma$,
	$$d^{\nabla^{\Sigma}} A^{\Sigma}(X,Y,Z)= g\big(\Rm(X,Y)N^{\Sigma},Z\big),$$
	and
	$$\delta^{\nabla^{\Sigma}} A^{\Sigma}(X) =\Ric\big(N^{\Sigma},X\big) - dH(X).$$
	Denoting $A_0^{\tilde{\Sigma}}$ the trace-free part of $A^{\tilde{\Sigma}}$ for $g^{\tilde{\Sigma}}$, if $\tilde{\Sigma}$ is a constant mean curvature hypersurface, by the above formulas, we even have
	\begin{equation}
	    d^{\nabla^{\tilde{\Sigma}}} A_0^{\tilde{\Sigma}}(X,Y,Z)= g\big(\Rm(X,Y)N^{\tilde{\Sigma}},Z\big),\label{equatino seconde forme fondamentale 1}
	\end{equation}
	and
	\begin{equation}
	    \delta^{\nabla^{\tilde{\Sigma}}} A_0^{\tilde{\Sigma}}(X) =\Ric\big(N^{\tilde{\Sigma}},X\big).\label{equatino seconde forme fondamentale 2}
	\end{equation}
	The right hand sides being well controlled, we control $A_0^{\tilde{\Sigma}}$ by noticing that the operator $d^{\nabla^{\tilde{\Sigma}}}+\delta^{\nabla^{\tilde{\Sigma}}}$ is elliptic.
	\begin{lem}\label{elliptique et noyau sur la sphère}
		The operator $d^{\nabla^{\tilde{\Sigma}}}+\delta^{\nabla^{\tilde{\Sigma}}}$ acting on symmetric $2$-tensors is elliptic. 
		Moreover, for any metric $g$ sufficiently $C^0$-close to $g^{\mathbb{S}^{n-1}}$, the operator $d^{\nabla^{\mathbb{S}^{n-1}}}+\delta^{\nabla^{\mathbb{S}^{n-1}}}$ is injective on the symmetric $2$-tensors which are traceless for $g$.
	\end{lem}
	\begin{proof}	
	    It is well-known that the operator $d^{\nabla}+\delta^{\nabla}$ is elliptic. According to \cite[12.69]{bes}, the associated Laplacian: $-\hat{\Delta}:= \delta^{\nabla} d^{\nabla}+d^{\nabla} \delta^{\nabla}$ takes the following form for a symmetric $2$-tensor $h$,
		$$-\hat{\Delta}^{\tilde{\Sigma}}h = (\nabla^{\tilde{\Sigma}})^*\nabla^{\tilde{\Sigma}}h -\mathring{\R}^{\tilde{\Sigma}}h+ h\circ \Ric^{\tilde{\Sigma}}, $$ 
		where $\mathring{\R}$ is the action of curvature on symmetric $2$-tensors: for an orthonormal basis $(e_i)_i$,
		$$\mathring{\R}(h)(X,Y)= \sum_i h\big(\Rm(e_i,X)Y,e_i\big),$$
		or $\mathring{\R}(h)_{ij}:= \Rm_{iklj} h_{mn}g^{km}g^{ln}$ in local coordinates. Hence, on the unit sphere of dimension $n-1$, we have $\Rm^{\mathbb{S}^{n-1}}_{iklj} = g_{ij}g_{kl}-g_{il}g_{kj}$, and therefore $\mathring{\R}^{\mathbb{S}^{n-1}}(h)_{ij}=tr(h)g_{ij} - h_{ij}$, and $\alpha\circ \Ric^{\mathbb{S}^{n-1}} = (n-2)\alpha$. This gives
		$$-\hat{\Delta}^{\mathbb{S}^{n-1}}h = (\nabla^{\mathbb{S}^{n-1}})^*\nabla^{\mathbb{S}^{n-1}}h +(n-1)h - tr(h)\mathrm{Id}.$$
		
		Let $h_0$ be a $C^{1,\alpha}(\mathbb{S}^{n-1})$ symmetric $2$-tensor whose trace for $g^{\mathbb{S}^{n-1}}$ vanishes and which is in the kernel of the operator. The trace of $h_0$ vanishes so we have,
		$$ (\nabla^{\mathbb{S}^{n-1}})^*\nabla^{\mathbb{S}^{n-1}}h_0 +(n-1)h_0 = 0.$$
		Integrating by parts against $h_0$ yields
		\begin{equation}
		\|\nabla h_0\|_{L^2}^2+ (n-1)\|h_0\|_{L^2}^2 =0,\label{L2 sphere}
		\end{equation}
		and finally $h_0 = 0$. The operator $d^{\nabla^{\mathbb{S}^{n-1}}}+\delta^{\nabla^{\mathbb{S}^{n-1}}}$ is therefore injective on symmetric $2$-tensors whose trace for $g^{\mathbb{S}^{n-1}}$ vanishes.
		\\
		
		Let us now consider a metric $g$ close to $g^{\mathbb{S}^{n-1}}$, that is, assume that there exists $\epsilon>0$ which we will choose small enough such that
		$$\|g-g^{\mathbb{S}^{n-1}}\|_{C^0(\mathbb{S}^{n-1})}\leq \epsilon.$$
		We want to show that the operator $d^{\nabla^{\mathbb{S}^{n-1}}}+\delta^{\nabla^{\mathbb{S}^{n-1}}}$ is also invertible on tensors whose trace vanishes for $g$. Let $h$ be a symmetric $2$-tensor  which is traceless for $g$. By proximity of the metrics, for $\epsilon$ small enough, there exists $C=C(n)$ for which
		\begin{equation}
		|tr^{\mathbb{S}^{n-1}} h|_{\mathbb{S}^{n-1}} \leqslant C\epsilon |h|_{\mathbb{S}^{n-1}}.\label{trace}
		\end{equation}
		We still have
		$$-\hat{\Delta}^{\mathbb{S}^{n-1}}h = (\nabla^{\mathbb{S}^{n-1}})^*\nabla^{\mathbb{S}^{n-1}}h +(n-1)h - tr^{\mathbb{S}^{n-1}}(h)\textup{Id},$$
		and thanks to \eqref{trace}, if $ \big(d^{\nabla^{\mathbb{S}^{n-1}}}+\delta^{\nabla^{\mathbb{S}^{n-1}}}\big) h = 0$, then
		$$0 \geqslant \|\nabla^{\mathbb{S}^{n-1}}h\|^2 +(n-1-C\epsilon)\|h\|^2,$$
		which gives $h =0$ if $\epsilon$ is small enough to have $n-1>C\epsilon$. We finally conclude that $d^{\nabla^{\mathbb{S}^{n-1}}}+\delta^{\nabla^{\mathbb{S}^{n-1}}}$ is injective on symmetric $2$-tensors whose trace for $g$ vanishes.
	\end{proof}
	
	\begin{cor}\label{est d + delta}
        For any $1\leqslant p<+\infty$ and $l\in \mathbb{N}$, there exists $\epsilon>0$ such that if $\Sigma$ is a hypersurface of a Riemannian manifold $(M^n,g)$ for which there exists $\Gamma$ a finite subgroup of $SO(n)$, and a diffeomorphism $\Phi : \mathbb{S}^{n-1}\slash\Gamma\to \tilde{\Sigma} $ such that we have 
	    $$\left\|\Phi^*g^{\tilde{\Sigma}} - g^{\mathbb{S}^{n-1}\slash\Gamma}\right\|_{W^{l+1,p}(\mathbb{S}^{n-1}\slash\Gamma)} \leqslant \epsilon,$$
	    then, there exists $C>0$ such that for any symmetric $2$-tensor $u$ on $\tilde{\Sigma}$ whose $g^{\tilde{\Sigma}}$-trace vanishes, we have,
	    $$\|u\|_{W^{l+1,p}(\mathbb{S}^{n-1}\slash\Gamma)}\leqslant C\Big(\big\|\delta^{\nabla^{\tilde{\Sigma}}}u\big\|_{W^{l,p}(\mathbb{S}^{n-1}\slash\Gamma)}+\big\|d^{\nabla^{\tilde{\Sigma}}}u\big\|_{W^{l,p}(\mathbb{S}^{n-1}\slash\Gamma)}\Big).$$
	\end{cor}
	\begin{proof}
	We have the following control: for all $p\geqslant 1$, there exists a constant $C>0$ such that for all symmetric $2$-tensor $u$ on $\mathbb{S}^{n-1}$ we have
	$$\Big\| \Phi^*\big(d^{\nabla^{\tilde{\Sigma}}}+\delta^{\nabla^{\tilde{\Sigma}}}\big)u-\big(d^{\nabla^{\mathbb{S}^{n-1}}}+\delta^{\nabla^{\mathbb{S}^{n-1}}}\big)u\Big\|_{L^{p}(\mathbb{S}^{n-1}\slash\Gamma)} \leqslant C\epsilon\left\|u\right\|_{W^{1,p}(\mathbb{S}^{n-1}\slash\Gamma)}.$$
	
	The conclusion comes then from estimates for injective Fredholm operators which are consequences of the open mapping theorem between Banach spaces.
	\end{proof}

\begin{proof}[Proof of Proposition \ref{contrôle CMC}]
	Let $0<\alpha<1$, and $p > \frac{n}{1-\alpha}$. Then, $W^{1,p}(g^{\tilde{\Sigma}})$ embeds continuously in $C^\alpha(g^{\tilde{\Sigma}})$, which means that there exists $C_1 =C_1(p,\alpha,n)>0$ such that for all symmetric $2$-tensor $u$, we have:
	$$\|u\|_{C^\alpha(g^{\tilde{\Sigma}})}\leq C_1\|u\|_{W^{1,p}(g^{\tilde{\Sigma}})}.$$
	
	Define on $\mathbb{S}^{n-1}$ the operators, $P:=\delta^{\nabla^{\mathbb{S}^{n-1}}}+d^{\nabla^{\mathbb{S}^{n-1}}}$ and $P':= \Phi^*\big( \delta^{\nabla^{\tilde{\Sigma}}}+d^{\nabla^{\tilde{\Sigma}}}\big)$ between the Banach spaces $W^{1,p}(\mathbb{S}^{n-1})$ and $L^p(\mathbb{S}^{n-1})$. According to Lemma \ref{elliptique et noyau sur la sphère}, $P$ is injective on traceless symmetric $2$-tensors. Thanks to the expressions \eqref{equatino seconde forme fondamentale 1} and \eqref{equatino seconde forme fondamentale 2} of the operator $\delta^\nabla+d^\nabla$ applied to $A^{\tilde{\Sigma}}_0$, by Corollary \ref{est d + delta}, we have,	
	\begin{align*}
	\big\|A_0^{\tilde{\Sigma}}\big\|_{W^{1,p}(g^{\tilde{\Sigma}})}&\;\leqslant C\left(\big\|\delta^{\nabla^{\tilde{\Sigma}}}A_0^{\tilde{\Sigma}}\big\|_{L^{p}(g^{\tilde{\Sigma}})}+\big\|d^{\nabla^{\tilde{\Sigma}}}A_0^{\tilde{\Sigma}}\big\|_{L^{p}(g^{\tilde{\Sigma}})}\right) \\
	&\;= C\left(\big\| \Rm(.,.)N^{\tilde{\Sigma}}\big\|_{L^{p}(g^{\tilde{\Sigma}})} + \big\|\Ric(N^{\tilde{\Sigma}}) \big\|_{L^{p}(g^{\tilde{\Sigma}})}\right)
	\end{align*}
	Since the hypersurface has bounded volume (actually close to the volume of $\mathbb{S}^{n-1}\slash \Gamma$), there exists $C_2 = C_2(p,n)>0$ such that we have
	$$\|\Ric(N^{\tilde{\Sigma}}) \|_{L^{p}(g^{\tilde{\Sigma}})}\leqslant C_2 \| \Rm(.,.)N^{\tilde{\Sigma}}\|_{C^{0}(g^{\tilde{\Sigma}})} ,$$
	and
	$$\| \Rm(.,.)N^{\tilde{\Sigma}}\|_{L^{p}(g^{\tilde{\Sigma}})}\leqslant C_2 \| \Rm(.,.)N^{\tilde{\Sigma}}\|_{C^{0}(g^{\tilde{\Sigma}})}.$$
	Hence, by continuous embedding of $W^{1,p}(g^{\tilde{\Sigma}})$ in $C^\alpha(g^{\tilde{\Sigma}})$, we have
	$$\|A_0^{\tilde{\Sigma}}\|_{C^\alpha(g^{\tilde{\Sigma}})}\leqslant 2C_1C_2C \| \Rm(.,.)N^{\tilde{\Sigma}}\|_{C^{0}(g^{\tilde{\Sigma}})},$$
	which is the stated result for $k=0$.
	\\
	
	Let us finally explain how to control the higher levels of regularity by a bootstrap argument once we have found optimal coordinates on our hypersurface $\tilde{\Sigma}$. Assume that we have the following controls for $l\in \mathbb{N}$, on the hypersurface $\tilde{\Sigma}$: there exists $C>0$ such that,
	\begin{enumerate}
	\item $\|\Rm(g)\|_{C^{l+1}(g)}\leqslant \epsilon$,
	\item $\big\|A^{\tilde{\Sigma}} - g^{\tilde{\Sigma}}\big\|_{C^{l,\alpha}(g^{\tilde{\Sigma}})}\leqslant C\|\Rm(g)\|_{C^{l}(g)}$,
	\item there exists a diffeomorphism $\Psi:\mathbb{S}^{n-1}\slash\Gamma\to \tilde{\Sigma}$ with $$\left\|\Psi^*g^{\tilde{\Sigma}} - g^{\mathbb{S}^{n-1}\slash\Gamma}\right\|_{C^{l+1,\alpha}(\mathbb{S}^{n-1}\slash\Gamma)} \leqslant C\|\Rm(g)\|_{C^{l}(g)}.$$
	\end{enumerate}
	Notice that at this point, we have proven the estimates for $l=0$. Gauss equation gives:
	\begin{align}
	 g(\Rm^{\tilde{\Sigma}}(X,Y)Z,W) =&\; \;  g(\Rm(X,Y)Z,W)  \nonumber\\
	 &\;+  A^{\tilde{\Sigma}}(X,Z)A^{\tilde{\Sigma}}(Y,W)-A^{\tilde{\Sigma}}(Y,Z)A^{\tilde{\Sigma}}(X,W).\label{gauss eq}
	\end{align}
	And controls 1 and 2 above imply therefore that there exists $C=C(\alpha,l,n)>0$ which gives
	$$\big\|\Ric(g^{\tilde{\Sigma}})-(n-1)g^{\tilde{\Sigma}}\big\|_{C^{l,\alpha}(g^{\tilde{\Sigma}})}\leqslant C \|\Rm(g)\|_{C^{l+1}(g)},$$
	and we can use Proposition \ref{contrôles optimaux sphere} proven in the appendix.
	\begin{prop} \label{contrôles optimaux sphere}
Let $(M,g_0)$ be an $n$-dimensional manifold with sectional curvatures equal to $1$, and $k\in\mathbb{N}$.
    Then, there exists $\delta>0$ and $C>0$, such that if a metric $g$ on $M$ satisfies:
$$\|g-g_0\|_{C^{1,\alpha}(g_0)}\leqslant \delta$$
    then, there exists a diffeomorphism $\Phi: M\to M$, such that we have,
    $$\|\Phi^*g-g_0\|_{C^{k+1,\alpha}(g_0)}\leqslant C\big\|\Ric(g)-(n-1)g\big\|_{C^{k-1,\alpha}(g_0)}.$$
    \end{prop}
	Hence, in our case, there exists a diffeomorphism (independent of $l$) $\tilde{\Psi}_{\tilde{\Sigma}}: \mathbb{S}^{n-1}\slash\Gamma \to \tilde{\Sigma}$ with
	$$\left\|\tilde{\Psi}_{\tilde{\Sigma}}^*g^{\tilde{\Sigma}} - g^{\mathbb{S}^{n-1}\slash\Gamma}\right\|_{C^{l+2,\alpha}(\mathbb{S}^{n-1}\slash\Gamma)} \leqslant C\|\Rm(g)\|_{C^{l+1}(g)}.$$
	This diffeomorphism combined to the controls in $C^{l+1}$-norm on the curvature imply by Corollary \ref{est d + delta} (just like in the proof with $k=0$) that there exists a constant $C'=C'(l+1,\alpha,n)>0$ such that $$\big\|A^{\tilde{\Sigma}} - g^{\tilde{\Sigma}}\big\|_{C^{l+1,\alpha}(g^{\tilde{\Sigma}})}\leqslant C'\|\Rm(g)\|_{C^{l+1}(g)}.$$
	We therefore have the property for $l+1$ and the control on the metric $\tilde{\Psi}_{\tilde{\Sigma}}^*g^{\tilde{\Sigma}}$ improves thanks to Proposition \ref{contrôles optimaux sphere} and we can iterate this argument for all $l\leqslant k-1$.
\end{proof}

\section{Perturbation of hypersurfaces with almost constant second fundamental form}

During all of our construction, without loss of generality by rescaling, we will always reduce our construction to the case where the mean curvature is constant equal to $n-1$. Let us first show that we can perturb a hypersurface with almost constant second fundamental form to a constant mean curvature hypersurface.
\begin{prop}\label{Existence pert CMC}
For all $n\in \mathbb{N}$ and $0<r_0< \frac{1}{6}$, there exists $C>0$ and $0<\epsilon<\frac{r_0}{C}$ such that for any hypersurface $\Sigma$ diffeomorphic to $\mathbb{S}^{n-1}\slash\Gamma$, for $\Gamma \neq \{\mathrm{Id}\}$, of a manifold $(M,g)$ satisfying:
	\begin{enumerate}
		\item the normal injectivity radius of $\Sigma$ and the injectivity radius of the points of $\Sigma$ in $M$ are bounded below by $3r_0$,
		\item $\big\|A^\Sigma - g^\Sigma\big\|_{C^{\alpha}(g^\Sigma)}\leqslant \epsilon$,
		\item on the annulus $\{x,d^M_g(x,\Sigma)<3r_0\}$, we have $\|\Rm(g)\|_{C^2(g)}<\epsilon$,
	\end{enumerate}
	
	Then, there exists a unique function $w$ satisfying $\|w\|_{C^{2,\alpha}(g^\Sigma)}\leqslant C\epsilon$ and such that we have $$H^{\Sigma(w)} \equiv n-1,$$ where we denoted $\Sigma(w):= \{\gamma_x(w(x)), x\in \Sigma\}$, where $\gamma_x$ is the geodesic with $\gamma_x(0) =x$, and $\gamma'_x(0)$ is the unit outwards normal to $\Sigma$. We moreover have the following control,
	$$\|w\|_{C^{2,\alpha}(g^\Sigma)}\leqslant C \big\|H^{\Sigma}-(n-1)\big\|_{C^\alpha(g^\Sigma)}.$$
\end{prop}

\begin{rem}
    By adapting the arguments of Section 3, we know that a hypersurface satisfying the assumptions of Proposition \ref{Existence pert CMC} is arbitrarily close (up to a diffeomorphism) in the $C^{1,\alpha}$-topology to the round $\mathbb{S}^{n-1}\slash\Gamma$.
\end{rem}

\begin{rem}
	Recall that we defined our Hölder norms at the scale of the injectivity radius, see Definition \ref{norme holder}. In our applications, the controls on the injectivity radius will be consequences of Lemma \ref{contrôle du rayon d'injectivité}, and the different curvature controls will be consequences of the previous section. 
\end{rem}

    In general, the construction of spheres with constant mean curvature can be tricky (see for example \cite{ye,px}) because it is achieved by constructing a right-inverse for the operator $\Delta_{\mathbb{S}^{n-1}}+(n-1)$, where $\Delta_{\mathbb{S}^{n-1}}$ is the Laplace-Beltrami operator (with nonpositive eigenvalues) which is not invertible. In our case, our hypersurfaces are not close to $\mathbb{S}^{n-1}$, but to $\mathbb{S}^{n-1}\slash \Gamma$ for $\Gamma \neq \{\mathrm{Id}\}$. The crucial difference it makes is that $-(n-1)$ is not an eigenvalue anymore (see \cite{ce} for the use of the same remark in the context of asymptotically conical manifolds). Indeed, the eigenfunctions with eigenvalue $-(n-1)$ are the restrictions of nonzero linear functions of $\mathbb{R}^n$ to the sphere which cannot be $\Gamma$-invariant.
	
	\begin{rem}
	    Geometrically, this comes from the fact that given $\mathbb{S}^{n-1}\subset \mathbb{R}^n$, the translations provide constant mean curvature perturbations of the spheres. In the case of $\mathbb{S}^{n-1}\slash\Gamma\subset \mathbb{R}^n\slash\Gamma$, the center cannot move as it is the fixed point of the group action.
	\end{rem}

	     \paragraph{Normal geodesic coordinates.}	
        Let us first give some controls on the geometry of the hypersurfaces equidistant to a given hypersurface $\tilde{\Sigma}_1$.
    \begin{lem}\label{seconde forme fondamentale equidistantes}
		For all $n\in\mathbb{N}$ and $0<r_0<\frac{1}{6}$, there exists $C>0$ and $0<\epsilon<\frac{r_0}{C}$ such that if a hypersurface $\tilde{\Sigma}_1$ of a manifold $(M,g)$ satisfies:
	\begin{itemize}
		\item the normal injectivity radius of $\tilde{\Sigma}_1$ and the injectivity radius of points of $\tilde{\Sigma}_1$ in $M$ are bounded below by $3r_0$,
		\item $H^{\tilde{\Sigma}_1} \equiv n-1$, and $\|A^{\tilde{\Sigma}_1}-g^{\tilde{\Sigma}_1}\|_{C^{\alpha}(g^{\tilde{\Sigma}_1})}<\epsilon$,
		\item on the annulus $\{x,d^M_g(x,\Sigma)<3r_0\}$, we have $\|\Rm\|_{C^2(g)}\leqslant \epsilon$,
	\end{itemize}
	then,
		\begin{align}
	\Big\|\frac{A^{\Sigma_s}}{s}-\frac{g^{\Sigma_s}}{s^2}\Big\|_{C^{\alpha}\big(\frac{g^{\Sigma_s}}{s^2}\big)}&\;\leqslant C\big(\|A^{\tilde{\Sigma}_1}-g^{\tilde{\Sigma}_1}\|_{C^{\alpha}g^{\tilde{\Sigma}_1}}+ |s-1|\|\Rm\|_{C^\alpha(g)} + |s-1|^2\|\Rm\|_{C^2(g)}\big)\nonumber\\
	&\;\leqslant C(1 + r_0^2)\epsilon,\label{contrôle A equidist}
	\end{align}
and
	\begin{equation}
	\Big\|\frac{H^{\Sigma_s}}{s}-\frac{n-1}{s^2}\Big\|_{C^{\alpha}\big(\frac{g^{\Sigma_s}}{s^2}\big)}\leqslant C\big(|s-1|\|\Ric\|_{C^\alpha(g)} + |s-1|^2 \|\Rm\|_{C^2(g)}\big)\leqslant C\big( r_0 + r_0^2\big)\epsilon. \label{contrôle H equidist}
	\end{equation}
	\end{lem}
	\begin{proof}
	Given a hypersurface $\Sigma_1$ whose normal injectivity radius is larger than $3r_0>0$, we define a vector field $ N$ on the region $\{x,d^M_g(x,\Sigma)<3r_0\}$ by setting $N(x)$, for $x\in\Sigma_s$, as the outwards normal vector of the hypersurface $\Sigma_s$. 
	
	Denoting, $S(X):= \nabla_X N$, for $X$ a vector field on $M$, we have
	\begin{equation}
	(\nabla_NS)(X) = - (S^2(X)+\Rm(X,N)N). \label{riccati seconde forme}    
	\end{equation}
    Assuming $\|S_{|\Sigma_1}-\mathrm{Id}\|_{C^\alpha(g^{\Sigma_1})}<\epsilon$ for $\epsilon>0$ small enough, we can follow $S$ along the normal geodesics, and integrating the previous Riccati equation \eqref{riccati seconde forme}, we find that there exits $C>0$ such that for $|s-1| \leqslant r_0 \leqslant \frac{1}{6}$,  $$\Big\|\frac{S_{|\Sigma_s}}{s} - \frac{\mathrm{Id}}{s^2} \Big\|_{C^\alpha(g^{\Sigma_s})} \leqslant \|S_{|\Sigma_1}-\mathrm{Id}\|_{C^\alpha(g^{\Sigma_1})} + C|s-1|\|\Rm_g\|_{C^\alpha(g)} + C|s-1|^2 \|\Rm_g\|_{C^2(g)}.$$ Denoting $\kappa_i(s)$ the eigenvalues of the second fundamental form of $\Sigma_s$, and $(e_i(s))$ an associated orthonormal basis of eigenvectors, if we project \eqref{riccati seconde forme} on $e_i(s)$, we get:
	$$\kappa_i'(s) =  - \left(\kappa_i(s)^2 + g\big(\Rm(N^{s},e_i(s)) e_i(s),N^s\big)\right)$$
    and finally, by summing on $i$, $H'(s) = - \big(|A|^2(s)+ \Ric(N^{s},N^{s})\big)$. With the initial condition $H(1) \equiv n-1$, we finally have the second stated estimate.
	\end{proof}

    Let us now construct coordinates in the neighborhood of a hypersurface satisfying the assumptions of Proposition \ref{Existence pert CMC}.
    \begin{lem}\label{coord geod}
        Let $\Sigma$ be a hypersurface satisfying the assumptions of Proposition \ref{Existence pert CMC}.
        Then, there exists a diffeomorphism $\Phi : A_e(1-2r_0,1+2r_0)\to \{x,d(x,\Sigma)<2r_0\}$ such that we have for $C>0$,
        $$\|\Phi^*g-g_e\|_{C^{1,\alpha}(g_e)}\leqslant C\epsilon,$$
        where $\epsilon>0$ is the constant of Proposition \ref{Existence pert CMC}.
    \end{lem}
    \begin{proof}
        Let us consider the diffeomorphism $F_g : \Sigma\times [-2r_0,2r_0]$ obtained from the outwards normal geodesic of $\Sigma$.
        
        We then use our controls on the second fundamental form and the ambient curvature and integrate the Riccati equation satisfied by the Jacobi fields generated by the infinitesimal change of starting point on the hypersurface $\Sigma$. We therefore obtain like in Lemma \ref{seconde forme fondamentale equidistantes} the following control for $C= C(n,v_0)>0$, by noting $s$ the parameter of the geodesics vanishing on $\Sigma$ :
        \begin{equation}
            \|F_g^*g-(ds^2+s^2g^\Sigma)\|_{C^{1,\alpha}}\leqslant C\epsilon.\label{contrôle coordonnées Fg*g}
        \end{equation}
        By the assumptions of Proposition \ref{Existence pert CMC}, up to choosing $\epsilon$ arbitrarily small, the second fundamental form is arbitrarily close to being constant and the ambient curvature arbitrarily small. In particular, according to Gauss equation \eqref{gauss eq} together with Lemma \ref{sec pincée ch gro} and Proposition \ref{contrôles optimaux sphere}, this implies that there exists $\Gamma\subset SO(4)$ such that, up to reparametrizing $\Sigma$, we have for some $C=C(n,v_0)>0$, 
        $$\|g^\Sigma-g^{\mathbb{S}^{n-1}\slash\Gamma}\|_{C^{1,\alpha}}\leqslant C\epsilon.$$
        Combining this with the control \eqref{contrôle coordonnées Fg*g}, we get the stated result.
    \end{proof}
	
	\paragraph{Normal perturbations of the hypersurface.}
	
	Consider normal perturbations of a hypersurface $\Sigma$ of the form $\Sigma(w):= \{\gamma_x(w(x)), x\in \Sigma\}$. The computations of \cite{px} adapted to our situation yield the following development of the mean curvature.
	
	\begin{lem}\label{Courbure moyenne pert orth}
		For any $0<r_0<\frac{1}{6}$ and any smooth hypersurface $\Sigma\subset M$ whose injectivity radius and normal injectivity radius are larger than $3r_0$, there exists $\epsilon>0$ such that for all $w: \Sigma \to \mathbb{R}$ such that $\| w\|_{C^{2,\alpha}(g^\Sigma)}\leqslant \epsilon$, the mean curvature of $\Sigma(w)$ has the development:
		$$H(\Sigma(w)) = H(\Sigma) - J_{\Sigma,M} w + Q_{\Sigma,M}(w),$$
		where $$J_{\Sigma,M}:= \Delta_\Sigma + |A^\Sigma|^2 + \emph{Ric}^M(N,N)$$ is the Jacobi operator of the hypersurface, and where $Q_{\Sigma,M}$ is such that for a constant $C>0$ depending on the $C^2$-norm of the curvature and the $C^1$-norm of the second fundamental form and $r_0$, we have :
		$$ \|Q_{\Sigma,M}(w) - Q_{\Sigma,M}(w')\|_{C^\alpha(g^\Sigma)}\leqslant  C \| w-w' \|_{C^{2,\alpha}(g^\Sigma)} \big(\| w\|_{C^{2,\alpha}(g^\Sigma)}+\| w' \|_{C^{2,\alpha}(g^\Sigma)}\big).$$
	\end{lem}
	\begin{proof}
	    
	\end{proof}

Proposition \ref{Existence pert CMC} is then a consequence of the following Lemma applied to $\Sigma$.
	
	\begin{lem}\label{existence perturbation}
	For any $r_0>0$, there exists $\epsilon>0$ such that if a metric $g$ on the annulus $A_e(1-r_0,1+r_0)\subset \mathbb{R}^n\slash\Gamma$ for $\Gamma\subset SO(n)$ with $\Gamma\neq \{\textup{Id}\}$ satisfies
	\begin{itemize}
	    \item $\|g-g_e\|_{C^{1,\alpha}(g_e)}\leqslant \epsilon$,
	    \item on $S$ the unit sphere centered at zero in $S\subset\mathbb{R}^n\slash\Gamma$, we have $$\|A_g^{S}-g_{|S}\|_{C^\alpha(S)}\leqslant \epsilon ,$$
	    \item on $A_e(1-r_0,1+r_0)$, we have
	    $$\|\Rm_g\|_{C^4(g_e)}\leqslant \epsilon,$$
	\end{itemize}
	Then, there exists $\epsilon_1=\epsilon_1(r_0,\epsilon,n,\alpha)>0$ such that there exists a unique solution $w$ satisfiying $\|w\|_{C^{2,\alpha}(S)}\leqslant \epsilon_1$ to the equation
		$$H_g(S(w)) \equiv n-1,$$
		where $H_g$ is the mean curvature for $g$ and where $S(w)= \{\gamma_x(w(x)),x\in S\}$ for $\gamma_x$ the geodesic starting $x$ and outwards normal to $S$.
		
		We moreover have the control:
		$$ \|w\|_{C^{2,\alpha}(S)}\leqslant C \left\|H_g(S) - (n-1)\right\|_{C^\alpha(S)}.$$
	\end{lem}
	\begin{proof}
	    This is a consequence of the following quantitative version of the inverse function theorem proven by Banach fixed point theorem.
		\begin{lem}\label{fonctions inverses}
				Let $\Phi: E\to F$, be a smooth map between Banach spaces and let $Q:= \Phi - \Phi(0)- d_0\Phi$.
				
				Assume that there exist $q>0$, $r_0>0$ and $c>0$ such that
				\begin{enumerate}
				\item for all $x$ and $y$ in $B(0,r_0)$, we have the following control on the nonlinear terms
				$$\|Q(x)-Q(y)\|\leqslant q (\|x\|+\|y\|)\|x-y\|.$$
				\item the linearization $d_0\Phi$ is an isomorphism, and more precisely, we have
				$$\|(d_0\Phi)^{-1}\|\leq c.$$
				\end{enumerate}
				
				If $r\leqslant \min\Big(r_0,\frac{1}{2qc}\Big)$ and $\|\Phi(0)\| \leqslant \frac{r}{2c}$, then, the equation $\Phi(x) =0$ admits a unique solution in $B(0,r)$.
		\end{lem}
		We will equip $S$ with the Hölder norms given by the restriction of $g_e$ on $S$.	Let us consider the operator $\Phi : C^{2,\alpha}(S) \to C^{0,\alpha}(S)$ defined by
		$$\Phi :w\in C^{2,\alpha}(S) \mapsto H_g(S(w))-(n-1)\in C^{0,\alpha}(S),$$
		and show that it satisfies the assumptions of Lemma \ref{fonctions inverses}, with $\Phi(0) = H_g(S)-(n-1) $, $d_0\Phi = J_{S,g}$.
		\begin{rem}
		    In this proof, we will precise the metric instead of the manifold for the Jacobi operator and the nonlinear terms of the mean curvature variations.
		\end{rem}
		\begin{enumerate}
			\item The operator $J_{S,g}$ is invertible and has a bounded inverse for $\epsilon>0$ small enough. Indeed, $-(n-1)$ is not an eigenvalue of the Laplacian because $\Gamma\neq\{\textup{Id}\}$, therefore
			$J_{S,g_e}$ is invertible with a bounded inverse. Moreover, by proximity of the metrics and the control of the $C^\alpha(S)$-norm of the second fundamental form and the $C^4$-norm of the ambient curvature, there exists $C_1>0$ such that, 
			$$\|J_{S,g_e} - J_{S,g}\|\leqslant C_1\epsilon,$$ for the operator norm from $C^{2,\alpha}(S)$ to $C^{\alpha}(S)$. The operator $J_{S,g}$ is therefore invertible with a bounded inverse for $\epsilon$ small enough since this is an open condition.
			
			\item The nonlinear terms are controlled in a quadratic way. There exists a constant $C>0$ depending on $r_0, \alpha$ and $\epsilon$  :
			$$ \|Q_{S,g}(w) - Q_{S,g}(w')\|_{C^\alpha}\leqslant  C \| w-w' \|_{C^{2,\alpha}(S)}(\|w\|_{C^{2,\alpha}}+\|w'\|_{C^{2,\alpha}(S)}). $$
			
			\item Since the second fundamental form is almost constant, we have:
			$$\left\|H_g(S) - (n-1)\right\|_{C^\alpha(S)} \leqslant \epsilon.$$
		\end{enumerate}
		Therefore, there exists a unique solution $w$ to the equation $\Phi(w)=0$ such that $\|w\|_{C^{2,\alpha}(S)}\leqslant \epsilon_1 := \min\Big(\frac{1}{2CC_1}, 2C_1\epsilon\Big)$. This solution moreover satisfies for $C'>0$, $$\|w\|_{C^{2,\alpha}(S)}\leqslant C' \|H_g(S)-(n-1)\|_{C^{\alpha}(S)}.$$
\end{proof}		

\begin{proof}[Proof of Proposition \ref{Existence pert CMC}]
    According to Lemma \ref{coord geod}, under the assumptions of Proposition \ref{Existence pert CMC}, then, there exists a diffeomorphism $\Phi : A_e(1-2r_0,1+2r_0)\to \{x,d(x,\Sigma)<2r_0\}$ such that we have for $C>0$,
        \begin{itemize}
            \item the metrics are close, $$\|\Phi^*g-g_e\|_{C^{1,\alpha}(g_e)}\leqslant C\epsilon,$$
            \item for the unit-sphere centered at zero denoted $S\subset\mathbb{R}^n\slash\Gamma$, we have $$\|A_{\Phi^*g}^{S}-(\Phi^*g)_{|S}\|_{C^\alpha}\leqslant \epsilon ,$$
	    \item for the curvature on $A_e(1-r_0,1+r_0)$, we have
	    $$\|\Rm_{\Phi^*g}\|_{C^4}\leqslant \epsilon,$$
        \end{itemize}
        where $\epsilon>0$ is the constant in the assumptions of Proposition \ref{Existence pert CMC}, for which the metric $\Phi^*g$ satisfies the assumptions of Lemma \ref{existence perturbation} with $S = \Phi^*\Sigma$ the unit-sphere of $\mathbb{R}^n\slash\Gamma$. Applying Lemma \ref{existence perturbation}, we obtain the statement of Proposition \ref{Existence pert CMC}.
\end{proof}

	For $\Sigma$ satisfying the assumptions of Proposition \ref{Existence pert CMC}, let us define $\tilde{\Sigma} := \Sigma(w)$, where $w$
	is the unique solution given by the proof of Proposition \ref{Existence pert CMC}. 

\section{Construction of a foliation of the neck regions by constant mean curvature hypersurfaces}
Let us now foliate our neck regions by constant mean curvature hypersurfaces.
\subsection{Local foliation by constant mean curvature hypersurfaces}
We start by constructing our foliation locally.

\begin{prop}[Local foliation]\label{feuilletage local}
    For all $n\in\mathbb{N}$ and $0<r_0<\frac{1}{6}$, there exists $C>0$ and $0<\epsilon<\frac{r_0}{C}$ such that if a hypersurface $\tilde{\Sigma}_1$ diffeomorphic to $\mathbb{S}^{n-1}\slash\Gamma$, for $\Gamma \neq \{\mathrm{Id}\}$, of a manifold $(M,g)$ satisfies:
	\begin{itemize}
		\item the normal injectivity radius of $\tilde{\Sigma}_1$ and the injectivity radius of the points $\tilde{\Sigma}_1$ in $M$ are bounded below by $3r_0$,
		\item $H^{\tilde{\Sigma}_1} \equiv n-1$, and $\|A^{\tilde{\Sigma}_1}-g^{\tilde{\Sigma}_1}\|_{C^{\alpha}(g^{\tilde{\Sigma}_1})}<\epsilon$,
		\item on the annulus $\{x,d^M_g(x,\Sigma)<3r_0\}$, we have $\|\Rm_g\|_{C^2(g)}\leqslant \epsilon$.
	\end{itemize}
	
	Then, there exists a family of hypersurfaces $\tilde{\Sigma}_s$, $s\in \Big[\frac{1}{1+2r_0}, 1+2r_0\Big]$ with constant mean curvature $\frac{n-1}{s}$ foliating a region of the annulus $\tilde{A}\Big(\frac{1}{1+2r_0}, 1+2r_0\Big)$ satisfying:	
	$$ \left\{x,\;d^M_g(x,\Sigma)<r_0\right\}\subset \tilde{A}\Big(\frac{1}{1+2r_0}, 1+2r_0\Big)  \subset \left\{x,\;d^M_g(x,\Sigma)<3r_0\right\}. $$
\end{prop}
    The construction we propose consists in perturbing normally the level sets of the distance function to the hypersurface $\tilde{\Sigma}_1$ in constant mean curvature hypersurfaces by Proposition \ref{Existence pert CMC}.
    \\
    
	Let us define the equidistant hypersurfaces $\Sigma_s:= \{\gamma_x(s-1), x\in \tilde{\Sigma}_1 \} $ for $s>0$ small enough where $\gamma_x$ is the outwards normal geodesic to the hypersurface $\tilde{\Sigma}$ starting at $x$. These hypersurfaces foliate the annulus $\{x,d^M_g(x,\Sigma)<3r_0\}$ by definition of the normal injectivity and their geometry is well controlled around $\tilde{\Sigma}_1$.	
	
	\begin{proof}[Proof of Proposition \ref{feuilletage local}]	
	Thanks to the controls \eqref{contrôle H equidist}, and \eqref{contrôle A equidist}, we can use Lemma \ref{existence perturbation} to perturb normally each hypersurface $\Sigma_s$ to a constant mean curvature hypersurface $\tilde{\Sigma}_s$. 
	Hence, there exists a family of hypersurfaces $(\tilde{\Sigma}_s)_{s\in \left[\frac{1}{1+2r_0}, 1+2r_0\right]}$ with mean curvature constant equal to $\frac{n-1}{s}$ for each $s$. These hypersurfaces are defined from a \emph{unique} function $w_s$ from $\tilde{\Sigma}_1$ to $\mathbb{R}$ by Proposition \ref{Existence pert CMC},
	$$\tilde{\Sigma}_s = \Sigma(s(1+w_s)):= \{\gamma_x(s(1+w_s(x))), x\in \tilde{\Sigma}_1\},$$
	we moreover have the control
	$$\|w_s\|_{C^{2,\alpha}(g^{\tilde{\Sigma}_1})}\leqslant C\epsilon(|s-1| + |s-1|^2).$$

    Let us now show that this family actually foliates the annulus. More precisely, we will show that given $s$ and $s'\in [\frac{1}{1+2r_0}, 1+2r_0]$, close enough and such that $s<s'$, $\tilde{\Sigma}_{s'}$ is a normal perturbation of $\tilde{\Sigma}_{s}$ by a \emph{strictly positive} function $w_s^{s'}: \tilde{\Sigma}_s\to\mathbb{R}$. They are indeed normal perturbation of each other because they are small (in $C^{2,\alpha}(g^{\tilde{\Sigma}_1})$-norm) normal perturbations of the equidistant hypersurfaces from $\tilde{\Sigma}_1$.
	
	By Lemma \ref{existence perturbation}, for $s$ and $s'$ close enough, the function $w_s^{s'}$ satisfies $\|w_s^{s'}\|_{C^{2,\alpha}((g^{\tilde{\Sigma}_s}))}\leqslant 2 |s'-s|$ and is a solution to $$-J_{\tilde{\Sigma}_{s},M}(w_s^{s'}) =  \frac{n-1}{s'}-\frac{n-1}{s} + Q_{\tilde{\Sigma}_{s},M}(w_s^{s'}),$$
 where there exists $C=C(\epsilon,n,\alpha,r_0)>0$ independent of $s$ such that for any functions $w$, and $w'$, we have $\|Q_{\tilde{\Sigma}_{s},M}(w) - Q_{\tilde{\Sigma}_{s},M}(w')\|_{C^\alpha(g^{\tilde{\Sigma}_s})}\leqslant C \| w-w' \|_{C^{2,\alpha}}(\|w\|_{C^{2,\alpha}(g^{\tilde{\Sigma}_s})}+\|w'\|_{C^{2,\alpha}(g^{\tilde{\Sigma}_s})})$.
	
	We therefore get $\|Q_{\tilde{\Sigma}_{s},M}(w_s^{s'})\|_{C^\alpha(g^{\tilde{\Sigma}_s})}\leqslant 4C |s'-s|^2 $, and since $s\tilde{\Sigma}_{s}$ is close to $\mathbb{S}^{n-1}\slash\Gamma$, for $|s'-s|$ small enough, $J_{\tilde{\Sigma}_{s},M}$ is close to $\Delta^{\mathbb{S}^{n-1}\slash\Gamma}  + n-1$, and we have:
	$$\Big\|-\big(\Delta^{\mathbb{S}^{n-1}\slash\Gamma}  + n-1\big) (w_s^{s'}-(s'-s)) \Big\|_{C^{\alpha}(g^{\tilde{\Sigma}_s})} \leqslant C \big((s'-s)^2 + \epsilon (s'-s)\big) .$$
	Since $-\Delta^{\mathbb{S}^{n-1}\slash\Gamma}+n-1$ is invertible, there exists $C=C(\alpha, n)>0$ such that $$\|w_s^{s'}-(s'-s)\|_{C^{2,\alpha}(g^{\tilde{\Sigma}_s})}\leqslant C \big((s'-s)^2 + \epsilon (s'-s)\big).$$ Finally, we have $w_s^{s'} = (s'-s) + \mathcal{O}\big((s'-s)^2 + \epsilon (s'-s)\big)$, and in particular, for $s'-s>0$ and $\epsilon$ small enough (depending on $n$ and $\alpha$), we have $w_s^{s'} >0$.
	
	This implies that for all $s,s'\in [\frac{1}{1+2r_0}, 1+2r_0]$, the hypersurfaces $\tilde{\Sigma}_s$ and $\tilde{\Sigma}_{s'}$ are all disjoint and foliate the annulus. Indeed, we have just seen that for $s<s'$ and $|s-s'|$ small enough, the hypersurface $\tilde{\Sigma}_{s'}$ is strictly included in one side of $\tilde{\Sigma}_{s}$. Recalling that for all $s$, $\tilde{\Sigma}_s = \tilde{\Sigma}_1(s(1+w_s))$, this means that for all $x$, the function $s\mapsto (s(1+w_s(x)))$ is strictly increasing, and therefore that we indeed have a foliation of the annulus.
\end{proof}
	
	\begin{rem}
	    It was not directly possible to ensure that two hypersurfaces $\tilde{\Sigma}_s$ and $\tilde{\Sigma}_{s'}$ do not intersect for $s,s'\in [\frac{1}{1+2r_0}, 1+2r_0]$ from our previous controls. Indeed, the control $\|w_s\|_{C^{2,\alpha}(g^{\tilde{\Sigma}_1})}\leqslant C(\epsilon|s-1| +\epsilon |s-1|^2)$ is not enough to rule out $w_s(x) = w_{s'}(x')$ when $s-1$ is much larger than $s-s'$.
	\end{rem}

\subsection{Global foliation of the neck regions.}

Let us come back to our applications and consider an Einstein manifold $(M,g)$ close to an Einstein orbifold $(M_o,g_o)$ in the Gromov-Hausdorff sense. According to Proposition \ref{description dégénérescence}, this manifold is $\epsilon$-approximated by a tree of Ricci-flat ALE orbifolds desingularizing the orbifold $(M_o,g_o)$. We have satisfactory coordinates on the regions $\mathcal{M}_o(\epsilon)$ and $\mathcal{N}_{k}(\epsilon)$ and want to find good coordinates in the annuli $\mathcal{A}_{k}(t,\epsilon)$. They are included in metric annuli $A(\rho_1,\rho_2)$ for $\rho_1:= \frac{1}{8}\epsilon^{-1} T_{k}^{\frac{1}{2}}$ and $\rho_2:= 8 \epsilon T_{j}^{\frac{1}{2}}$ on which we have the following properties:
\begin{enumerate}
	\item the curvature is controlled by \cite[Proposition 3]{ban}: for all $l\in\mathbb{N}$, there exists $C>0$ and $\beta_1,\beta_2>0$ depending on $g_o$ and the $g_{b_j}$ such that at distance $\rho\in \big[\frac{\rho_1}{2},2\rho_2]$ from $p$, we have
	\begin{equation}
	    \rho^{2+l} \big|\nabla^l \Rm\big|\leqslant \eta(\rho):= C \epsilon\left[\left(\frac{\rho_1}{\rho}\right)^{\beta_1}+\left(\frac{\rho}{\rho_2}\right)^{\beta_2}\right],\label{contrRm}
	\end{equation}
	\begin{rem}
	    In \cite{ban}, the result is only stated for $l=0$, but the regularity of Einstein metrics, and in particular the fact that the curvature is harmonic (for the Hodge Laplacian) yields the estimate for higher derivatives by local elliptic estimates.
	\end{rem}
	\item by Proposition \ref{Premières coordonnées}, there exists $\delta(\epsilon)$ such that $\delta(\epsilon)\to 0$ as $\epsilon\to 0$ and for any $2\rho_1<\rho<\frac{1}{4}\rho_2$, there exists a region $\hat{A}\big(\rho,2\rho\big)$ satisfiying
	$$A\big((1+\delta)\rho,(2-\delta)\rho\big) \subset \hat{A}\big(\rho,2\rho\big)\subset A\big((1-\delta)\rho,(2+\delta)\rho\big),$$
	$\Gamma$ a finite subgroup of $SO(4)$ and a diffeomorphism $$\Phi_\rho : A_{e}(1,2)\subset\mathbb{R}^4\slash\Gamma\to \hat{A}\big(\rho,2\rho\big)\subset M,$$
	for which we have,
	$$\Big\|\frac{\Phi_\rho^*g}{\rho^2}-g_e\Big\|_{C^{l+2}(A_{e}(1,2))}\leqslant \delta.$$
	
	\item by Lemma \ref{hypersurfaces presque sphériques}, there exists a hypersurface $\Sigma_{\rho_1}$ whose second fundamental form satisfies $\big\|\frac{A_g^{\Sigma_{\rho_1}}}{\rho_1}-\frac{g}{\rho_1^2}\big\|_{C^{l+1}(\frac{g}{\rho_1^2})}\leqslant \delta$,

	\item by Lemma \ref{contrôle du rayon d'injectivité} and Lemma \ref{injectivité normale} there exists $r_0>0$ such that the injectivity radius at each point $x$ is bounded below by $r_0 d(x,p)$, and the normal injectivity radius to $\Sigma_{\rho_1}$ by Lemma \ref{hypersurfaces presque sphériques} is bounded below by $r_0\rho_1$.
	\end{enumerate}
	
The control \eqref{contrRm} means that in the middle of the annulus, we have a much better control than initially expected since the classical $\epsilon$-regularity theorems \cite{gao,bkn,and} would only give $\rho^2|\Rm|\leqslant C \|\Rm\|_{L^2(A(\rho_1,\rho_2))}$. From the point of view of analysis, the difference is substantial as it allows one to do analysis in weighted Hölder spaces, and this will be heavily used in \cite{ozu2}.

\begin{rem}
    Analogous estimates hold in any dimension $n$ by assuming some much less natural $L^\frac{n}{2}$-controls on the Riemann curvature.
\end{rem}
	
	The main difficulty in constructing good coordinates here is that the interior radius, $\rho_1$, is neglectible when compared to the radius at which our controls are optimal, $\bar{\rho}:= (\rho_1^{\beta_1} \rho_2^{\beta_2})^{\frac{1}{\beta_1+\beta_2}}$ which minimizes $\eta$. Constant mean curvature hypersurfaces are convenient in this context as they can be constructed in the entire zone, and can be controlled by Proposition \ref{contrôle CMC}.

\begin{prop}\label{feuilletage}
For all $n\in\mathbb{N}$, there exists $\epsilon_0>0$, and $C>0$ such that for any $0<\epsilon<\epsilon_0$, we have the following property. Let $\mathcal{A}_{k}(t,\epsilon)$ be an intermediate annulus of an Einstein metric which is $\epsilon$-approximated by a naïve desingularization, and assume that it is included in a metric annulus $A(4\rho_1,\frac{1}{4}\rho_2)$ with $\int_{A(\rho_1,\rho_2)}|\Rm|^2dv< \epsilon^2$. Then, there exists a foliation of a region $\tilde{A}\big(2\rho_1,\frac{1}{2}\rho_2\big)$ bounded by two hypersurfaces of mean curvature respectively constant equal to  $\frac{n-1}{2\rho_1}$ and $\frac{n-1}{\frac{1}{2}\rho_2}$, with 	
	$$ \mathcal{A}_{k}(t,\epsilon)\subset A\Big((2+C\epsilon)\rho_1,\Big(\frac{1}{2}-C\epsilon\Big)\rho_2\Big)\subset \tilde{A}\big(2\rho_1,\frac{1}{2}\rho_2\big) \subset A\Big((2-C\epsilon)\rho_1,\Big(\frac{1}{2}+C\epsilon\Big)\rho_2\Big),$$
	by hypersurfaces denoted $\tilde{\Sigma}_s$ whose mean curvature equals $\frac{n-1}{s}$.
\end{prop}

\begin{proof}
Considering $\Phi_{2\rho_1}$, the diffeomorphism of Proposition \ref{Premières coordonnées} at scale $2\rho_1$, we define the hypersurface $\Sigma_{2\rho_1}:=\Phi_{2\rho_1,*}S\subset M$, where $S$ is the unit-sphere centered at zero of $\mathbb{R}^4\slash\Gamma$. According to Lemmata \ref{hypersurfaces presque sphériques} and \ref{injectivité normale}, and Corollary \ref{pert normale de chaque côté}, we satisfy the assumptions of Lemma \ref{existence perturbation}. We can therefore perturb $ \Sigma_{2\rho_1}$ to a hypersurface $\tilde{\Sigma}_{2\rho_1}$ whose mean curvature is constant equal to $ \frac{n-1}{2\rho_1} $. 

According to Proposition \ref{feuilletage local}, there exists a foliation by hypersurfaces $\tilde{\Sigma}_s$ whose mean curvature is constant equal to $\frac{n-1}{s}$ for all $2\rho_1<s<(1+2r_0)2\rho_1$. Moreover, for any $s$, there exists a constant $C=C(n,\alpha)>0$ such that if we recall the notation
$$\eta(s) := C \epsilon\left[\left(\frac{\rho_1}{s}\right)^{\beta_1}+\left(\frac{s}{\rho_2}\right)^{\beta_2}\right],$$ 
we have the following \emph{a priori} controls on the $\tilde{\Sigma}_s$ thanks to Proposition \ref{contrôle CMC}:
	\begin{itemize}
		\item the normal injectivity radius of $\tilde{\Sigma}_s$ and the injectivity radius of the points of $\tilde{\Sigma}_s$ in $M$ are bounded below by $3sr_0$ according to Lemmata \ref{contrôle du rayon d'injectivité} and \ref{injectivité normale},
		\item $H^{\tilde{\Sigma}_1} \equiv \frac{n-1}{s}$, and $\|\frac{A^{\tilde{\Sigma}_s}}{s}-\frac{g^{\tilde{\Sigma}_s}}{s^2}\|_{C^{\alpha}}<C\eta(s)$,
		\item on the annulus $\{x,d^M_g(x,\tilde{\Sigma}_1)<3sr_0\}$, we have $\|\Rm_g\|_{C^4(g)}\leqslant \frac{C\eta(s)}{s^2}$.
	\end{itemize}
	We can therefore apply Proposition \ref{feuilletage local} from the hypersurface $\tilde{\Sigma}_{2(1+2r_0)\rho_1}$ in order to obtain a foliation by hypersurfaces with constant mean curvature equal to $\frac{n-1}{s}$ up to $s = 2(1+2r_0)^2\rho_1$. Once again, the \emph{a priori} controls of Proposition \ref{contrôle CMC} of order $C\eta(s)$ and the injectivity and normal injectivity controls of Lemmata \ref{contrôle du rayon d'injectivité} and \ref{injectivité normale} let us iterate this process in order to construct hypersurfaces $\tilde{\Sigma}_s$ with mean curvature constant equal to $\frac{n-1}{s}$ up to $s= 2(1+2r_0)^k\rho_1$ for any $k\in \mathbb{N}$ with $2(1+2r_0)^k\rho_1<\frac{1}{2}\rho_2$ which is the stated result.
\end{proof}

\section{Construction of coordinates and control on the metric}\label{Construction coordonnées et controle}

To construct a diffeomorphism, we will use the constant mean curvature hypersurface where our estimates are optimal and follow the gradient lines of the function $s:= \frac{n-1}{H(\tilde{\Sigma}_s)}$ whose level sets are the hypersurfaces $\tilde{\Sigma}_s$. The main result is the following Proposition:

\begin{prop}\label{bonnes coord}
	Under the same assumptions as Proposition \ref{feuilletage}, there exists $\beta_1>0$, $\beta_2>0$, $\Gamma\neq\{\mathrm{Id}\}$ a finite subgroup of $SO(4)$ such that for all $l\in \mathbb{N}$, there exists $C>0$, and a diffeomorphism: 	
	$$\Phi: A_{e}\Big(2\rho_1,\frac{1}{2}\rho_2\Big) \to \tilde{A}\Big(2\rho_1,\frac{1}{2}\rho_2\Big) ,$$
	such that for all $\rho\in [2\rho_1,\frac{1}{4}\rho_2]$, denoting $\Phi_\rho:= \Phi\circ \phi_\rho: A_{e}(1,2)\to \tilde{A}\Big(2\rho_1,\frac{1}{2}\rho_2\Big) ,$ where $\phi_\rho$ is the homothetic transformation of ratio $\rho$ on $\mathbb{R}^4\slash\Gamma$, we have:	
	$$\Big\|\frac{\Phi^*_\rho g}{\rho^2}-g_{e}\Big\|_{C^{l}(A_{e}(1,2))}\leqslant C \epsilon \left(\Big(\frac{\rho_1}{\rho}\Big)^{\beta_1}+\Big(\frac{\rho}{\rho_2}\Big)^{\beta_2}\right).$$ 
\end{prop}

\begin{rem}
    Our proofs again work in any dimension $n\in \mathbb{N}$ for metrics with a region close to a flat cone $\mathbb{R}^n\slash\Gamma$ (for $\Gamma\neq \{\mathrm{Id}\}$) where the curvature is small in $L^\frac{n}{2}$.
\end{rem}

\subsection{Coordinates based on our constant mean curvature foliation}
In the foliation constructed in Proposition \ref{feuilletage}, one particular hypersurface is better controlled than the others at its scale. It is $\tilde{\Sigma}_{\bar{\rho}}$ where $\eta(\bar{\rho})= \min_{\rho_1<\rho<\rho_2}\eta(\rho)$, which satisfies the following properties by \eqref{contrRm} and Proposition \ref{contrôle CMC}:
\begin{enumerate}
    \item there exists a diffeomorphism $\phi: \mathbb{S}^{n-1}\slash\Gamma \to \tilde{\Sigma}_{\bar{\rho}}$ such that: for all $k\in\mathbb{N}$ and $0<\alpha<1$, there exists $C=C(k,\alpha,n,r_0)>0$ for which,
    $$ \Big\|\frac{\phi^*g^{\tilde{\Sigma}_{\bar{\rho}}}}{\bar{\rho}^2}-g^{\mathbb{S}^n\slash\Gamma} \Big\|_{C^{k,\alpha}(g^{\mathbb{S}^n\slash\Gamma})}< C \eta(\bar{\rho}),$$
    \item for all $k\in\mathbb{N}$ and $0<\alpha<1$, there exists $C=C(k,\alpha,n)>0$ such that
    $$\Big\| \frac{A^{\tilde{\Sigma}_{\bar{\rho}}}}{\bar{\rho}} - \frac{g^{\tilde{\Sigma}_{\bar{\rho}}} }{\bar{\rho}^2}\Big\|_{C^{k,\alpha}\big(\frac{g^{\tilde{\Sigma}_{\bar{\rho}}} }{\bar{\rho}^2}\big)} \leqslant C\eta(\bar{\rho}).$$
    \item for all $k\in\mathbb{N}$, there exists $C=C(k,n)>0$ such that on the annulus $[\frac{1}{2}\bar{\rho},2\bar{\rho}]$ we have,
    $$\bar{\rho}^{2+k}\big|\nabla^k \Rm\big|_g \leqslant C \eta(\bar{\rho}).$$
\end{enumerate}

We then define a diffeomorphism $\Phi: [\rho_1,\rho_2]\times \mathbb{S}^{n-1}\slash\Gamma \to \tilde{A}(\rho_1,\rho_2)$ in the following way: at $\bar{\rho}$,
$$\Phi(\bar{\rho},x):= \phi(x),$$
where $\phi$ is the diffeomorphism of the above point $1$, and for all $s\in [\rho_1,\rho_2]$,
$$\partial_s\Phi(s,x) = u_s(x) N^{\tilde{\Sigma}_s}\big(\Phi(s,x)\big) =  -\frac{\nabla s}{|\nabla s|^2},$$
where $u_s$ is such that $\Phi(s,\mathbb{S}^{n-1}\slash\Gamma) = \tilde{\Sigma}_s$. This is a diffeomorphism on its image since two curves following the vector field $-\frac{\nabla s}{|\nabla s|^2}$, which never vanishes, cannot intersect if they are not equal.

\subsection{Control of the metric in these coordinates}

In the coordinates given by $\Phi$, the metric has the following form,
$$\Phi^*g(s,x) = u_s(x)^2ds^2 + s^2 h_s(x),$$
where $h_s$ is a metric on $\mathbb{S}^{n-1}\slash\Gamma$.
\\

Let us take the following notations for the rest of the section in which we will always work on $\mathbb{R}^n\slash\Gamma$:
\begin{itemize}
	\item $\phi_s(x) = \Phi(s,x)$, which implies that $s^2h_s = \phi_s^*g^{\tilde{\Sigma}_s}$.
	\item $H(s) = \phi_s^*H^{\tilde{\Sigma}_s} \equiv \frac{n-1}{s}$,
	\item $A(s) = \phi_s^*A^{\tilde{\Sigma}_s}$,
	\item $N(s) = \phi_s^*N^{\tilde{\Sigma}_s} $
	\item $f_i(s) = d_{(s,.)}\Phi(0,e_i)$, where the $e_i$ form an orthonormal basis of $T_x\mathbb{S}^{n-1}\slash \Gamma$,
	\item $\nabla^s = \phi_s^*\nabla^{\tilde{\Sigma}_s}$,
	\item $\Delta^s = \phi_s^*\Delta^{\tilde{\Sigma}_s}$, where $\Delta^{\tilde{\Sigma}_s}$ is the Laplace-Beltrami operator associated to the metric $g^{\tilde{\Sigma}_s}$,
	\item $\textup{grad}^{s}$ is the gradient for the metric $s^2h_s$,
	\item $K(s)_i^j = g\big(\Rm_g(N(s),f_i(s))f_j(s),N(s)\big)$.
\end{itemize}
We have the following variations for the different geometric quantities.

	By \cite[Theorem 3.2]{hp}, given a variation $\partial_s\Phi(s,x) = u_s N^{\tilde{\Sigma}_s}$, we have the following variation formulas:
	\begin{equation}
	    \partial_s \big(s^2h_s\big) = 2 u_s A(s),\label{est 1 hp}
	\end{equation}
	\begin{equation}
	    \partial_s N(s) = \textup{grad}^{s} u_s,\label{est 2 hp}
	\end{equation}
	\begin{equation}
	    \partial_s A(s)_{ij} = - \nabla^s_i\nabla^s_j u_s + \big(\sum_kA(s)_{ik}A(s)_{j}^k+ K(s)_i^j\big),\label{est 3 hp}
	\end{equation}
	\begin{equation}
	   \partial_s H(s) = - \Big(\Delta^{s}+\big|A(s)\big|^2+\Ric(s)\big(N(s),N(s)\big)\Big)u_s.\label{est 4 hp}
	\end{equation}

Let us finally compare our Einstein metric to the flat metric in these coordinates.

\begin{prop}
 With the above notations, if we denote $g_e:=ds^2+s^2g^{\mathbb{S}^{n-1}}$, we have the following controls on $\Phi^*g$: for all $l\in\mathbb{N}$, there exists $C(l,n,v_0,D_0)>0$ such that for all $s\in \big[\rho_1,\frac{\rho_2}{2}\big]$, on the annulus $[s,2s]$, we have
 $$\|u-1\|_{C^l(g_e)}\leqslant C \eta(s),$$
 and
 $$\|h-g^{\mathbb{S}^{n-1}\slash\Gamma}\|_{C^l(g_e)}\leqslant C\eta(s),$$
 and therefore, for all $s\in[\rho_1,\frac{\rho_2}{2}]$, on the annulus of radii $s$ and $2s$,
 $$\|\Phi^*g-g_e\|_{C^l(g_e)} \leqslant C\eta(s).$$
\end{prop}
\begin{proof}
    All along the proof, $C$ will denote a positive constant that may vary from line to line, but which only depends on the constants $n$, $v_0$, $D_0$ and the order of regularity $l\in \mathbb{N}$, $0<\alpha<1$.

    Let us start by noting that the equation \eqref{est 4 hp} implies that for all $s$, $u_s$ satisfies the following equation
	\begin{equation}
		\frac{n-1}{s^2} = \Big(\Delta^{s}+\big|A(s)\big|^2+\Ric(s)\big(N(s),N(s)\big)\Big)u_s.\label{equation u}
	\end{equation}
	Indeed, by following the normal perturbation $u_s N^{\tilde{\Sigma}_s} =-\frac{\nabla s}{|\nabla s|^2}$, the variation of the mean curvature is exactly $-\frac{n-1}{s^2}$. In particular, for all $s$, we have
		\begin{equation}
		\Big(\Delta^{s}+\big|A(s)\big|^2+\Ric(s)\big(N(s),N(s)\big)\Big)(u_s-1) = \frac{n-1}{s^2}-\big|A(s)\big|^2-\Ric(s)\big(N(s),N(s)\big),\label{equation u-1}
	\end{equation}
	
	According to Proposition \ref{contrôle CMC}, while $\|h_s-g^{\mathbb{S}^{n-1}\slash\Gamma}\|_{C^{1,\alpha}(g^{\mathbb{S}^{n-1}\slash\Gamma})}$ is sufficiently small, there exists $C = C(l)>0$ such that for all $s$,
\begin{equation}
	\Big\|\frac{A(s)}{s}-h_s\Big\|_{C^{l,\alpha}(h_s)}\leqslant C \eta(s).\label{controle A}
\end{equation}		
	Since $s^2\big|\Ric(s)\big(N(s),N(s)\big)\big|<\eta(s)$, this implies that the right hand side of \eqref{equation u-1} satisfies: there exists $C=C(\alpha,n)>0$ such that for all $s$,
	$$\Big\|n-1-s^2\big|A(s)\big|^2-s^2\Ric(s)\big(N(s),N(s)\big)\Big\|_{C^{\alpha}(g^{\mathbb{S}^{n-1}\slash\Gamma})}\leqslant C\eta(s).$$	
	Likewise, for the left hand side, for any function $v: \mathbb{S}^{n-1}\slash\Gamma\to \mathbb{R}$ we have
	\begin{align*}
	\Big\|s^2\Big(\Delta^{s}+\big|A(s)\big|^2+\Ric(s)\big(N(s),N(s)\big)\Big)v - &\;\Big(\Delta^{\mathbb{S}^{n-1}\slash\Gamma}+n-1\Big) v \Big\|_{C^{\alpha}(g^{\mathbb{S}^{n-1}\slash\Gamma})}\\
	 &\;\leqslant C\eta(s) \|v\|_{C^{2,\alpha}(g^{\mathbb{S}^{n-1}\slash\Gamma})}.
	\end{align*}
	Since $\Delta^{\mathbb{S}^{n-1}\slash\Gamma}+n-1$ is invertible for $\Gamma\neq \{\mathrm{Id}\}$, we get by the inverse function theorem, Lemma \ref{fonctions inverses}, the following control: there exists $C=C(n,\alpha)>0$ such that for all $s$
	 \begin{equation}
	\|u_s-1\|_{C^{2,\alpha}(g^{\mathbb{S}^{n-1}\slash\Gamma})}\leqslant C\eta(s).\label{controle u-1}
	 \end{equation}
	 
	 Let us now show that for all $s\in [\rho_1,\rho_2]$, $\|h_s-g^{\mathbb{S}^{n-1}\slash\Gamma}\|_{C^{1,\alpha}(g^{\mathbb{S}^{n-1}\slash\Gamma})}$ remains small. We will actually show that there exists $C=C(n,\alpha)>0$ such that $\|h_s-g^{\mathbb{S}^{n-1}\slash\Gamma}\|_{C^{1,\alpha}(g^{\mathbb{S}^{n-1}\slash\Gamma})}\leqslant C \eta(s)$.  
	The expression \eqref{est 1 hp} can be rewritten $$\partial_s \big(h_s-g^{\mathbb{S}^{n-1}\slash\Gamma}\big) =\frac{2}{s}\Big(u\frac{A^s}{s}-h_s\Big),$$
	but by the control \eqref{controle u-1}, and Proposition \ref{contrôle CMC}, there exists $C=C(n,\alpha)>0$ such that we have $\Big\|u\frac{A^s}{s}-h_s\Big\|_{C^{1,\alpha}(g^{\mathbb{S}^{n-1}\slash\Gamma})}\leqslant C \eta(s)$. Recall that $$\eta(s):= C\epsilon\Big[\Big(\frac{\rho_1}{s}\Big)^{\beta_1}+\Big(\frac{s}{\rho_2}\Big)^{\beta_2} \Big],$$ to deduce that for $\bar{\rho}<s_0<\rho_2$, we have
	\begin{align*}
	\big\|h_{s_0}-g^{\mathbb{S}^{n-1}\slash\Gamma}\big\|_{C^{1,\alpha}(g^{\mathbb{S}^{n-1}\slash\Gamma})} \leqslant&\; C \Big(\eta(\bar{\rho})+ \int_{\bar{\rho}}^{s_0}\frac{2}{s} \eta(s) ds\Big)\\
	=&\; C\eta(\bar{\rho})+\Big(\frac{2C\epsilon}{\rho_2^{\beta_2}} \int_{\bar{\rho}}^{s_0} s^{\beta_2-1} ds\Big)\\
	=&\; C\epsilon\frac{\bar{\rho}^{\beta_2}}{\rho_2^{\beta_2}}+\Big(\frac{2C\epsilon}{\rho_2^{\beta_2}}(s_0^{\beta_2}-\bar{\rho}^{\beta_2}) \Big)\\
	\leqslant &\; C \eta(s_0).
	\end{align*}
	and similarly, for $\rho_1<s_0<\bar{\rho}$,
	$$\big\|h_{s_0}-g^{\mathbb{S}^{n-1}\slash\Gamma}\big\|_{C^{1,\alpha}(g^{\mathbb{S}^{n-1}\slash\Gamma})}\leqslant C \eta(s_0).$$
	
	To obtain controls on higher derivatives of $u$ and $h_s$, we use the other equalities of \cite[Theorem 3.2]{hp}: according to \eqref{est 2 hp}, and the control \eqref{controle u-1}, there exists $C =C(n,\alpha)$ such that for all $s$ we have
\begin{equation}
	\|\partial_s N(s)\|_{C^{1,\alpha}(h_s)}\leqslant C\eta(s)\label{controle N}
\end{equation}	
according to \eqref{est 3 hp} and the controls \eqref{controle u-1}, \eqref{controle A} and \eqref{controle N}, and since the $l$-th derivatives of the curvature are bounded by $s^{-2-l} \eta(s)$, there exists $C=C(n,\alpha)>0$ such that for all $s$, we have:
	\begin{equation}
	\Big\|\partial_s \Big(\frac{A(s)}{s}\Big)\Big\|_{C^{\alpha}(h_s)}\leqslant C\eta(s).\label{controle dA}
	\end{equation}
	
	Differentiating the equality \eqref{est 4 hp} with respect to $s$, and using the inequalities $\|\partial_s h_s\|_{C^{1,\alpha}(h_s)}\leqslant C\eta(s)$, and \eqref{controle dA}, 
	there exists $C=C(n,\alpha)>0$ such that for all $s$, we have:
	$$\Big\|\Big(\Delta^{\mathbb{S}^{n-1}\slash\Gamma}+n-1\Big)\partial_su_s\Big\|_{C^\alpha(h_s)} \leqslant C \eta(s).$$
	We can therefore conclude that for all $s$, $\|\partial_su_s \|_{C^{2,\alpha}(h_s)}\leqslant C\eta(s)$.
	
	Iterating this for higher derivatives in all directions, we obtain the stated controls.
\end{proof}

\subsection{Proximity between the Einstein metric and the naïve desingularization in weighted $C^k$-norm}

Now, to construct coordinates on the whole manifold $M$, we just "glue" the coordinates we have on each part together by using the local diffeomorphisms with the common asymptotic cone of each part. We obtain the following control for the metric.
\begin{thm}\label{dégénérescence et proximité C1alpha}
	Let $D_0,v_0>0$ and $(M^{\mathcal{E}}_i,g^\mathcal{E}_i)$ a sequence of Einstein manifolds satisfying
	\begin{itemize}
		\item the volume is bounded below by $v_0>0$,
		\item the diameter is bounded by $D_0$,
		\item the Ricci curvature is bounded $|\Ric|\leq 3$.
	\end{itemize}
	
	Then, there exists a subsequence with fixed topology $M=M_i^\mathcal{E}$ and a sequence of naïve desingularizations $(M,g^D_{t_i})$ of an Einstein orbifold $(M_o,g_o)$ such that: for all $l\in \mathbb{N}$ and $i$ large enough, there exists $C(l, v_0, D_0)>0$, $\beta_1(v_0, D_0)>0$ and $\beta_2(v_0, D_0)>0$, $\epsilon_i>0$, $\epsilon_i\to 0$, and a diffeomorphism $\Phi_i: M\to M$ satisfying
	\begin{enumerate}
	\item on $\mathcal{M}_o(\epsilon_i)$, we have
	$$\Big\|\Phi^*_ig^\mathcal{E}_i-g_o\Big\|_{C^l(g_o)}\leqslant C\epsilon_i.$$

	\item at the different scales $T_{i,j}$ (associated to $t_i$), on $\mathcal{N}_{j}(\epsilon_i)$ we have
	$$\Big\|\frac{\Phi^*_ig^\mathcal{E}_i}{T_{i,j}}-g_{b_j}\Big\|_{C^l(g_{b_j})} \leqslant C\epsilon_i.$$
	
	\item in the intermediate regions, $\mathcal{A}_{k}(t_i,\epsilon_i)$ included in a metric annulus of radii $\rho_1^{i,k}= \frac{1}{8}T_{i,k}^{\frac{1}{2}}\epsilon_i^{-1}$, $\rho_2^{i,k} = 8 T_{i,j}^{\frac{1}{2}}\epsilon_i$, for all $\rho_1^{i,k}\leqslant \rho \leqslant \frac{1}{2}\rho_2^{i,k}$, and denoting $\Phi_{i,\rho} = \Phi_i \circ \phi_\rho$, where $\phi_\rho$ is the homothetic transformation of ration $\rho$ on $\mathbb{R}^4\slash\Gamma_k$ with its flat metric $g_e$, we have
	$$\Big\|\frac{\Phi_{i,\rho}^*g^\mathcal{E}_i}{\rho^2}-g_e\Big\|_{C^l(A_e(1,2))} \leqslant C\epsilon_i\Big[\Big(\frac{\rho_1^{i,k}}{\rho}\Big)^{\beta_1}+\Big(\frac{\rho}{\rho_2^{i,k}}\Big)^{\beta_2}\Big].$$
	\end{enumerate}
\end{thm}

\begin{rem}
    This is indeed much better than a $C^\infty$ convergence on compacts to a flat cone as we have $\Big[\Big(\frac{\rho_1^{i,k}}{\rho}\Big)^{\beta_1}+\Big(\frac{\rho}{\rho_2^{i,k}}\Big)^{\beta_2}\Big]\leqslant 2$, and a decay in the annulus.
\end{rem}

\newpage

\appendix

\section{Optimal $C^{k,\alpha}$-proximity to a round sphere}

In this appendix, we are interested in metrics with pinched positive sectional curvatures and want to estimate how close to a round metric they have to be. To start with, a consequence of Cheeger-Gromov compactness is the following lemma, see \cite[Chapter 10]{pet} for instance.

\begin{lem}\label{presque rigidité sphere}\label{sec pincée ch gro}
	For all $n$, $0<\alpha<1$, $\delta>0$ and $v_0>0$, there exists $\epsilon>0$ such that if $(M^n,g)$ satisfies
	\begin{enumerate}
	\item $1-\epsilon<Sec(g)<1+\epsilon$,
	\item and $\vol(M)>v_0$,
	\end{enumerate} 
	then, there exists $\Gamma$ a finite subgroup of $SO(n+1)$ (whose order is bounded by a function of $v_0$), and a diffeomorphism $\Phi: \mathbb{S}^{n}\slash\Gamma \to M$ for which:
	$$\big\|\Phi^*g-g^{\mathbb{S}^n\slash\Gamma}\big\|_{C^{1,\alpha}(g^{\mathbb{S}^n\slash\Gamma})}\leqslant \delta,$$
	where $g^{\mathbb{S}^n\slash\Gamma}$ is a round metric on $\mathbb{S}^n\slash\Gamma$.
\end{lem}

In this appendix, we will prove the following proposition.

\begin{prop}\label{Controle précis distance}
Let $(M,g_0)$ be a Riemannian manifold of dimension $n$ with sectional curvatures equal to $1$, and $k\in\mathbb{N}$.
    Then, there exists $\delta>0$ and $C>0$, such that if a metric $g$ on $M$ satisfies:
   $$\|g-g_0\|_{C^{1,\alpha}(g_0)}\leqslant \delta,$$
    then, there exists a diffeomorphism $\Phi: M\to M$, for which $$\|\Phi^*g-g_0\|_{C^{k+1,\alpha}(g_0)}\leqslant C\big\|\Ric(g)-(n-1)g\big\|_{C^k(g_0)}.$$
    \end{prop}
As an immediate consequence of Lemma \ref{presque rigidité sphere} and Proposition \ref{Controle précis distance}, we get the following statement.
\begin{cor}
For all $n$, $0<\alpha<1$ and $v_0>0$, there exists $\epsilon_0>0$ and $C>0$ such that for any $0<\epsilon<\epsilon_0$, if $(M^n,g)$ satisfies:
	\begin{enumerate}
	\item $1-\epsilon<Sec(g)<1+\epsilon$,
	\item and $\vol(M,g)>v_0$,
	\end{enumerate} 
		then, there exists $\Gamma$ a finite subgroup of $SO(n+1)$, and a diffeomorphism $\Phi: \mathbb{S}^{n}\slash\Gamma \to M$ for which
		$$\|\Phi^*g-g_0\|_{C^{1,\alpha}(g_0)}\leqslant C \epsilon.$$
\end{cor}

\begin{rem}
	The crucial part for us is that the $\epsilon$ in the final estimate is the same as in the assumption $1-\epsilon<Sec(g)<1+\epsilon$.
\end{rem}

\begin{proof}[Proof of Proposition \ref{Controle précis distance}]
    Let $(M^n,g_0)$ be a Riemannian manifold whose sectional curvatures are constant equal to $1$. For $\lambda:= \frac{(n-2)(n-1)}{2}$, define the operator
    $$E_{g_0} (g):= \Ric(g)-\frac{\R(g)}{2}g+\lambda g + \delta^*_g\delta_{g_0} g.$$
    Then, following \cite[Section 5]{andsurv}, by \cite{ebi} and by the Bianchi identity, there exists $\delta>0$ such that $E_{g_0}^{-1}(\{0\})\cap B_{C^{1,\alpha}}(\delta)$ is exactly the set of Einstein metrics with constant $\Lambda$ which are in divergence-free gauge with respect to $g_0$ and have same volume.
    
    Let $g$ be a metric on $M$ such that $\|g-g_0\|_{C^{k+1,\alpha}(g_0)}\leqslant \delta$ for $\delta>0$ a constant which we will choose small enough in the rest of the proof. According to \cite{ebi}, for $\delta$ small enough, there exists a diffeomorphism $\Phi: M \to M$ such that  $$\delta_{g_0}\Phi^*g=0.$$    Hence, we have 
    \begin{equation}
        \|E_{g_0}(\Phi^*g)\|_{C^k(g_0)}\leqslant \big\|\Ric(g)-(n-1)g\big\|_{C^k(g_0)}.\label{contrôle E(phig) Ric(g)}
    \end{equation}
    Now, by \cite[Corollary 12.72 (Bourguignon, unpublished)]{bes}, the linearization of $E$ at $g_0$ is invertible (see again \cite[Section 5]{andsurv} for the adaptation to the operator $E$) and by the inverse function theorem Lemma \ref{fonctions inverses} we deduce that for $\delta$ small enough, and for all $1\leqslant p<+\infty$ and $k\in \mathbb{N}$ there exists $C>0$, such that we have
    $$\|\Phi^*g-g_0\|_{W^{k+2,p}(g_0)}\leqslant C\|E_{g_0}(\Phi^*g)\big\|_{W^{k,p}(g_0)},$$
    and finally, by \eqref{contrôle E(phig) Ric(g)}, we have
    $$\|\Phi^*g-g_0\|_{C^{k+1,\alpha}(g_0)}\leqslant C\|\Ric(g)-(n-1)g\big\|_{C^k(g_0)},$$
    because on the sphere, for $p$ large enough, $W^{k+2,p}(g_0)$ embeds continuously in $C^{k+1,\alpha}(g_0)$, and so does $C^k(g_0)$ in $W^{k,p}(g_0)$.
\end{proof}


%
%
%
%

\end{document}